\documentclass[a4paper,12pt,reqno]{amsart}
\usepackage{eurosym}
\usepackage[T1]{fontenc}
\usepackage{amsthm}
\usepackage{amsmath}
\usepackage{amssymb}
\usepackage{mathrsfs}
\usepackage{latexsym}
\usepackage{exscale}
\usepackage{geometry}
\usepackage{graphicx}
\usepackage[utf8]{inputenc}
\usepackage{color}
\usepackage[colorlinks=true, pdfstartview=FitV, linkcolor=blue, citecolor=red, urlcolor=blue]{hyperref}

\definecolor{red}{rgb}{1,0.1,0.1}
\definecolor{blue}{rgb}{0.1,0.1,1}
\definecolor{vb}{RGB}{160,32,240}

\numberwithin{equation}{section}
\newtheorem{theorem}{Theorem}[section]

\newtheorem{corollary}[theorem]{Corollary}

\newtheorem{definition}[theorem]{Definition}

\newtheorem{lemma}[theorem]{Lemma}

\newtheorem{proposition}[theorem]{Proposition}
\newtheorem{remark}[theorem]{Remark}

\calclayout
\allowdisplaybreaks

\newcommand{\RR}{\mathbb{R}}

\newcommand{\Xx}{\mathcal{X}}

\newcommand{\Ff}{\mathcal{F}}

\begin{document}
\title[Weighted $p(\cdot)$-Poincar\'e and Sobolev inequality]{Weighted $%
p(\cdot )$-Poincar\'{e} and Sobolev inequalities for vector fields
satisfiying H\"{o}rmander's condition and applications}
\author[L. A. Vallejos]{Lucas A. Vallejos}
\address{L. A. Vallejos\\
FaMAF \\
Universidad Nacional de C\'ordoba \\
CIEM (CONICET) \\
5000 C\'ordoba, Argentina}
\email{lucas.vallejos@unc.edu.ar}
\author[R.~E.~Vidal]{Ra\'ul E. Vidal}
\address{R.~E.~Vidal \\
FaMAF \\
Universidad Nacional de C\'ordoba \\
CIEM (CONICET) \\
5000 C\'ordoba, Argentina}
\email{vidal@famaf.unc.edu.ar}
\thanks{ The authors are partially supported by CONICET and SECYT-UNC}
\keywords{Poincar\'e inequalities, Weighted variable Sobolev spaces,
Carnot-Carath\'eodory spaces, Carnot group, Degenerate $p(\cdot)$-Laplacian.%
\\
\indent2020 {\ Mathematics Subject Classification: 35J92, 35R03, 42B35. }}

\begin{abstract}
In this paper we will establish different weighted Poincar\'{e} inequalities
with variable exponents on Carnot-Carath\'{e}odory spaces or Carnot groups.
We will use different techniques to obtain these inequalities. For vector
fields satisfying H\"{o}rmander's condition in variable non-isotropic
Sobolev spaces, we consider a weight in the variable Muckenhoupt class $%
A_{p(\cdot ),p^{\ast }(\cdot )}$, where the exponent $p(\cdot )$ satisfies
appropriate hypotheses, and in this case we obtain the first order weighted
Poincar\'{e} inequalities with variable exponents. In the case of Carnot
groups we also set up the higher order weighted Poincar\'{e} inequalities
with variable exponents. For these results the crucial part is proving the
boundedness of the fractional integral operator on Lebesgue spaces with
weighted and variable exponents on spaces of homogeneous type. Moreover,
using other techniques, we extend some of these results when the exponent
satisfies a jump condition and the weight is in a smaller Muckenhoupt class.

Finally, we will use these weighted Poincar\'{e} inequalities to establish
the existence and uniqueness of a minimizer to the Dirichlet energy integral
for a problem involving a degenerate $p(\cdot )$-Laplacian with zero
boundary values in Carnot groups.
\end{abstract}

\maketitle

%\date{July 16, 2007. \textit{Revised}: \today}

\section{Introduction and main results}

Let us consider $\mathcal{X}$ a Carnot-Carath\'{e}odory space, i.e. a
nonempty set with $X=(X_{1},\cdots ,X_{n_{1}})$ a family of infinitely
differentiable vector fields that satisfy the H\"{o}rmander's condition, see
Subsection \ref{sec2.1}. Poincar\'{e} type inequalities for vector fields
satisfying H\"{o}rmander's condition have been extensively studied since the
90s. A classic result is 
\begin{equation}\label{poincare}
\int_{\Omega }|f-f_{\Omega }|^{p}\,dy\leq C(\Omega )\int_{\Omega
}|Xf|^{p}\,dy,\qquad \qquad \text{for all }f\in W^{1,p}(\Omega ),
\end{equation}%
\begin{equation}\label{poincare'}
\int_{\Omega }|f|^{p}\,dy\leq C(\Omega )\int_{\Omega }|Xf|^{p}\,dy,\qquad
\qquad \text{for all }f\in W_{0}^{1,p}(\Omega ),
\end{equation}%
where $1\leq p<\infty $ and $\Omega \subset \mathcal{X}$ is a bounded open
set. If $\Omega $ is a ball of radius $r$ in the Carnot-Carath\'{e}odory
metric, then $C(\Omega )=Cr^{p}$, see \cite{J}, \cite{CDG} and \cite[Theorem
2.1]{DGP}.

Throughout the article we say that $\Omega \subset \mathcal{X}$ is a domain
if it is open, bounded and connected.

In \cite{FLW}, \cite{L} and \cite{LW} the authors prove different weighted
Poincar\'{e} inequalities for vector fields satisfying H\"{o}rmander's
condition, to obtain their results they prove different representation
formulas. Let $K$ be a compact subset of a domain $\Omega$. Let $\omega_1\in A_p$  and let $\omega_2$ a doubling weight, 
that they satisfying a appropriate balance condition. For $x\in K$ and $B=B(x,r)$, with $r<r_0$  
\begin{equation*}
\left(\frac1{\omega_2(B)}\int_{B}|f(y)-f_{B }|^{q}\omega_2(y)dy\right)^{1/q}\leq Cr\left(\frac1{\omega_1(B)}\int_{\Omega
}|Xf|^{p}\omega_1(y)dy\right)^{1/p},
\end{equation*}
for any Lipschitz continuous function f, where $C=C(\Omega,K,\omega_1,\omega_2)$ and $f_{B }=\frac1{\omega_2(B)}\int_B f\omega_2dx$.

 %For the definition of a weak Boman chain domain see Subsection \ref{sec2.1}.
In \cite{LLT}, in the case Carnot-Carath\'{e}odory spaces% and Carnot groups,  
using these representation
formulas and extrapolation results, has been proved                                
$p(\cdot )$ Poincar\'{e} inequalities for variable  Sobolev spaces. 
Let $\Omega$ be a weak Boman chain domain and let $p(\cdot )\in \mathcal{P}(\Omega )$ be variable exponent such that $1\leq p_- \leq p_+\leq Q$ 
suppose that the maximal operator $M$ is bounded on $L^{(p^*(\cdot)/Q')'}$. Then for any Lipschitz continuous function f
\begin{equation}\label{chinos}
\lVert (f-f_{\Omega })\rVert_{L^{p^{\ast }(\cdot )}(\Omega )}\leq
C\lVert Xf\rVert_{L^{p(\cdot )}(\Omega )},
\end{equation}
here $p^*(x)=\frac{Qp(x)}{Q-p(x)}$, $Q$ is the homogeneous dimension and $C=C(\Omega,p(\cdot))$.

Furthermore, in \cite{LLT}, for Carnot groups, the authors extend the result \ref{chinos} obtaining higher order Poincar\'e and Sobolev inequalities with variable exponent.

On the other hand, recently in \cite{BF} the authors consider
exponents that satisfy a jump condition and also prove a $p(\cdot )$ Poincar%
\'{e} inequality for variable exponent Sobolev spaces. In these works,
different hypotheses on the domain and the variable exponent are considered, 
generalizing the result given in \cite{LLT}, equation \ref{chinos}.
Moreover, in \cite{BF} the authors establish the existence and uniqueness of
a minimizer to the Dirichlet energy integral for the problem 

\begin{equation*}
\left\{ 
\begin{array}{rclc}
-\text{div}\left(|Xu(x)|^{p(x)-2}Xu(x)\right) & = & 
0,\qquad & x\in \Omega , \\[10pt] 
u(x) & = & v(x), & x\in \partial \Omega .%
\end{array}%
\right.,  
\end{equation*}
where $v:\partial \Omega \to \RR$ be a continuous function and  $\text{div}\left(|Xu(x)|^{p(x)-2}Xu(x)\right)$ is $p(\cdot )$-Laplacian in Carnot groups.

In this paper we prove weighted Poincar\'{e} and Sobolev inequalities with variable
exponents in non-isotropic Sobolev spaces associated to the vector fieds
satisfying H\"{o}rmander's condition on domains $\Omega $. %for Carnot-Carath\'{e}odory spaces with variable exponents. 
In this way we extend some of the
results given in \cite{LLT} and \cite{BF}.

%Our theorems establish the weighted Poincar\'{e} and Sobolev inequality with variable
%exponents in non-isotropic Sobolev spaces associated to the vector fieds
%satisfying H\"{o}rmander's condition on domains $\Omega $ satisfying the
%Boman chain condition, see Subsection \ref{sec2.1}. 

%Let $(\Xx,d,dx)$ be a Carnot-Carath\'eodory space. We 

%see Section \ref{sec2.1}. If $M$ is a Hardy-Littlewood maximal operators defined in $\Xx$, $\omega$ a weight in $\Xx$ and $p(\cdot)\in \mathcal{P}(\Xx)$ we say $(p(\cdot),\omega)$ is a $M$-pair, if $M$ is bounded in $L^{p(\cdot)}_\omega(\Xx)$, see Section \ref{sec2.2} for the definition of Sobolev an Lebesgue variable spaces and see Section \ref{sec2.3} for the definition of Hardy-Littlewood maximal operators and the variable Muckenhoupt clase $A_{p^{*}(\cdot),p(\cdot)}$. 

We now state our results. We denote by $Q$ the homogeneous dimension of $%
\mathcal{X}$ and $A_{p(\cdot ),p^{\ast }(\cdot )}$ is the weighted variable
Muckenhoupt class defined on $\mathcal{X}$. The definition of $M$-pair is
given in Subsection \ref{sec2.3}. 
%Using the representation formula given in Lemma \ref{RepresentationFormulaFLW}, we prove 

\begin{theorem}
\label{prin} Let $(\mathcal{X},d,dx)$ be a Carnot-Carath\'{e}odory space, $%
\Omega $ be a weak Boman chain domain, and $p(\cdot )\in \mathcal{P}(\Omega )
$ with $1<p^{-}\leq p^{+}<Q$. If $\omega \in A_{p(\cdot ),p^{\ast }(\cdot )}$
is such that $(p^{*}(\cdot)/Q', \omega^{Q'})$ is a $M$-pair,
then 
\begin{equation}
\lVert (f-f_{\Omega })\rVert _{L_{\omega}^{p^{\ast }(\cdot )}(\Omega )}\leq
C\lVert Xf\rVert _{L_{\omega}^{p(\cdot )}(\Omega )}
\end{equation}
for all Lipschitz continuous function $f$.
%\label{Poincar\U{c3}\U{a9}_Inequality}

Moreover, if the maximal operator is bounded on $L_{\omega ^{-Q^{\prime
}}}^{(p^{\ast }(\cdot )/Q^{\prime })^{\prime }}$, then the theorem holds for 
$p^{-}=1$.
\end{theorem}

For Carnot groups, using the boundedness of the fractional integral operator
on Lebesgue spaces with weight and variable exponent on spaces of
homogeneous type, we obtain the following theorems.

\begin{theorem}
\label{Poincare} Let $\Omega $ be a domain in a Carnot group $\mathbb{G}$
with homogeneous dimension $Q$. Suppose that $m$ and $j$ are integers with $%
0\leq j<m$ and $m-j\leq Q$, let $p(\cdot )\in \mathcal{P}(\mathbb{G})$
satisfy $1<p^{-}\leq p^{+}<\frac{Q}{m-j}$ and let $\sigma =(Q/m-j)^{\prime }$%
. If $\omega \in A_{p(\cdot ),p_{m,j}^{\ast }(\cdot )}$ is such that $%
(p_{m,j}^{\ast }(\cdot )/\sigma ,\omega ^{\sigma })$ is a $M$-pair, with $%
p_{m,j}^{\ast }(x)=\frac{Qp(x)}{Q-(m-j)p(x)}$, then 
\begin{equation}
\Vert X^{j}f\Vert _{L_{\omega }^{p_{m,j}^{\ast }(\cdot )}(\Omega )}\leq
C\Vert X^{m}(f)\Vert _{L_{\omega }^{p(\cdot )}(\Omega )}
\label{Sobolev_embedding}
\end{equation}%
for all $f\in W_{\omega ,0}^{m,p(\cdot )}(\Omega )$. %
%
%
%
%
%Also the inequality (\ref{Sobolev_embedding}) holds if $1\leq p^{-}\leq p^{+}\leq Q/m$ and the maximal operator is bounded on $L^{(p\ast (\cdot)/Q^{\prime })^{\prime }}(\omega ^{-Q^{\prime }})$.
\end{theorem}

\begin{theorem}
\label{prin2} Assume $\Omega $ is a weak Boman chain domain in a Carnot
group $\mathbb{G}$ with a central ball $B$ and homogeneous dimension $Q$.
Let $i$, $j$ and $m$ be integers with $0\leq j<i\leq m$ and $i-j\leq Q$. Let 
$p(\cdot )\in \mathcal{P}(\Omega )$ satisfy $1<p^{-}\leq p^{+}<\frac{Q}{i-j}$
and let $\sigma =(Q/i-j)^{\prime }$ . If $\omega \in A_{p(\cdot
),p_{i,j}^{\ast }(\cdot )}$ is such that $(p_{i,j}^{\ast }(\cdot )/\sigma
,\omega ^{\sigma })$ is a $M$-pair, with $p_{i,j}^{\ast }(x)=\frac{Qp(x)}{%
Q-(i-j)p(x)}$, then for every $f\in W_{\omega }^{m,p(\cdot )}(\Omega )$
there exists a polynomial $P_{m}\in \mathcal{P}_{m}$ such that 
\begin{equation*}
\Vert X^{j}(f-P_{m})\Vert _{L_{\omega }^{p_{i,j}^{\ast }(\cdot )}(\Omega
)}\leq C\Vert X^{m}(f)\Vert _{L_{\omega }^{p(\cdot )}(\Omega )}.
\end{equation*}
\end{theorem}

For the next Poincar\'e inequality we need that the variable exponent $%
p(\cdot)$ meets the jump condition, see Subsection \ref{sec2.2}.

\begin{theorem}
\label{Poincare2} Let $\Omega $ be a domain in a Carnot-Carath\'{e}odory
space $(\mathcal{X},d,dx)$. Let $p(\cdot )\in \mathcal{P}(\Omega )$, $%
1<p^{-}\leq p^{+}<Q$, such that satisfies the jump condition in $\Omega $
with $\delta >0$ and let $\omega $ be a weight such that $\omega \in
A_{p^{-},(p^{-})^{\ast }}$. Then 
\begin{equation*}
\lVert f\rVert _{L_{\omega }^{p(\cdot )}(\Omega )}\leq C\lVert Xf\rVert
_{L_{\omega }^{p(\cdot )}(\Omega )}
\end{equation*}%
for all $f\in W_{\omega ,0}^{1,p(\cdot )}(\Omega )$. The constant $C$
depends on $p$, $\Omega $, $\delta $, $\omega $ and $Q$.
\end{theorem}

\begin{remark}
The inequality in Theorem \ref{Poincare}, for the case $m=1$, implies the
inequality in Theorem \ref{Poincare2}, however both hypothesis are
different. On one hand in Theorem \ref{Poincare2} we have more variable
exponents $p(\cdot )$ because the jump condition does not necessarily imply
the continuity of $p(\cdot )$. In Theorem \ref{Poincare} the exponent $%
p(\cdot )$ has the property that $(p^{\ast }(\cdot )/\sigma ,\omega ^{\sigma
})$ is a $M$-pair, a condition enough for this to happen is that $p(\cdot )$
is locally Log-H\"{o}lder continuous. The last condition implies that $%
p(\cdot )$ is continuous, see Subsection \ref{sec2.3}. 
%In the other hand, the Theorem  \ref{Poincare2} is true for a smaller class of weights than the Theorem \ref{Poincare}, because if  $\omega \in A_{p^{-},(p^{-})^{\ast}}$ then $\omega \in A_{p(\cdot ),p^{\ast }(\cdot )}$.
\end{remark}

Finally we will give an application to a problem involving a degenerate $%
p(\cdot )$-Laplacian in Carnot groups. For these results we will always
consider that the weight $\omega $ meets the hypotheses of Theorem \ref%
{Poincare} or Theorem \ref{Poincare2}. In Subsection \ref{sec2.2} we prove
that the aforementioned hypothesis guarantee that $W_{\omega ,0}^{1,p(\cdot
)}(\Omega )$ is a Banach space.

Assume $\Omega $ is a domain in a Carnot group $\mathbb{G}$ and $\omega >0$
is a weight in $\Omega $. Let $A:\Omega \rightarrow \mathbb{R}^{n_{1}\times
n_{1}}$ be a measurable function such that, for each $x\in \Omega $, $A(x)$
is a symmetric matrix of dimension $n_{1}\times n_{1}$ and for any $\xi \in 
\mathbb{R}^{n_{1}}$ there exist $\eta _{1},\,\eta _{2}>0$ such that 
\begin{equation}
\omega ^{2}(x)\eta _{1}|\xi |^{2}\leq \left\langle A(x)\xi ,\xi
\right\rangle \leq \omega ^{2}(x)\eta _{2}|\xi |^{2}\qquad \text{for all }%
\,x\in \mathbb{G}.  \label{A}
\end{equation}%
For $u\in W_{\omega }^{1,p(\cdot )}(\Omega )$ we define the degenerate $%
p(\cdot )-$Laplacian by 
\begin{equation}
\mathcal{L}_{p(\cdot ),A}u(x):=\text{div}\left( \left\langle
A(x)Xu(x),Xu(x)\right\rangle ^{\frac{p(x)-2}{2}}A(x)Xu(x)\right) .
\label{dege}
\end{equation}%
If $p(x)=2$ then $A(x)=I$ is an identity matrix and $\mathcal{L}_{2,I}$ is a
classical SubLaplacian in Carnot groups.

For $f\in L_{\omega ^{-1}}^{p^{\prime }(\cdot )}(\Omega )$, we consider the
following second order elliptic differential problem with Dirichlet boundary
conditions

\begin{equation}
\left\{ 
\begin{array}{rclc}
-\mathcal{L}_{p(\cdot ),A}u(x)+|u(x)|^{p(x)-2}u(x)\omega (x)^{p(x)} & = & 
f(x),\qquad & x\in \Omega , \\[10pt] 
u(x) & = & 0, & x\in \partial \Omega .%
\end{array}%
\right.  \label{1}
\end{equation}%
A function $u\in W_{\omega ,0}^{1,p(\cdot )}(\Omega )$ is a weak solution of
problem \eqref{1} if%
\begin{align}
& \int_{\Omega }\left\langle A(x)Xu(x),Xu(x)\right\rangle ^{\frac{p(x)-2}{2}%
}\left\langle A(x)Xu(x),Xv(x)\right\rangle \,dx \\
& \qquad \qquad \quad +\int_{\Omega }|u(x)|^{p(x)-2}u(x)v(x)\omega
(x)^{p(x)}\,dx=\int_{\Omega }f(x)v(x)\,dx,  \notag  \label{2}
\end{align}%
for all $v\in W_{\omega ,0}^{1,p(\cdot )}(\Omega )$.

Associated with this problem we have the following energy functional 
\begin{equation}
\mathcal{F}(u)\!:=\!\int_{\Omega }\!\frac{\left\langle
A(x)Xu(x),Xu(x)\right\rangle ^{\frac{p(x)}{2}}}{p(x)}dx\!+\!\int_{\Omega }\!%
\frac{|u(x)\omega (x)|^{p(x)}}{p(x)}dx\!-\!\int_{\Omega }\!f(x)u(x)dx.
\label{func}
\end{equation}

We will get the following characterization.

\begin{theorem}
\label{carac} Let $\Omega $ be a domain in a Carnot group $\mathbb{G}$ and $%
p(\cdot )\in \mathcal{P}(\Omega )$, $1<p^{-}\leq p^{+}<Q$. Let $\omega >0$
be a weight and $A(x)$ be a matrix that meets the hypothesis \eqref{A}. Then 
$u\in W_{w,0}^{1,p(\cdot )}(\Omega )$ minimizes the energy functional $%
\mathcal{F}$ defined in \eqref{func} if and only if 
\begin{align*}
& \int_{\Omega }\left\langle A(x)Xu(x),Xu(x)\right\rangle ^{\frac{p(x)-2}{2}%
}\left\langle A(x)Xu(x),Xv(x)\right\rangle \,dx \\
& \qquad \qquad \qquad +\int_{\Omega }|u(x)|^{p(x)-2}u(x)v(x)\omega
(x)^{p(x)}\,dx=\int_{\Omega }f(x)v(x)\,dx\geq 0,
\end{align*}%
for every $v\in W_{w,0}^{1,p(\cdot )}(\Omega )$.
\end{theorem}

Concerning problem \eqref{1}, using variational methods we shall prove the
following existence and uniqueness result.

\begin{theorem}
\label{ext-unc} Let $\Omega $ be a domain in a Carnot group $\mathbb{G}$ and 
$p(\cdot )\in \mathcal{P}(\Omega )$, $1<p^{-}\leq p^{+}<Q$. Let $\omega >0$
be a weight and suppose the matrix $A(x)$ fulfills the hypothesis \eqref{A}.
Assume also that either $\omega $ and $p(\cdot )$ meet the hypothesis of
Theorem \ref{Poincare} with $m=1$ or $\omega $ and $p(\cdot )$ meet the
hypothesis of Theorem \ref{Poincare2}. Then there exists a unique minimizer $%
u\in W_{w,0}^{1,p(\cdot )}(\Omega )$ of the energy functional $\mathcal{F}$
defined in \eqref{func}, that is 
\begin{equation*}
\mathcal{F}(u)=\min_{v\in W_{w,0}^{1,p(\cdot )}(\Omega )}\mathcal{F}(v).
\end{equation*}%
Moreover, $u$ is the unique weak solution of the problem \eqref{1}.
\end{theorem}

The rest of the paper is organized as follows: In Subsection \ref{sec2.1} we
collect previous results about Carnot-Carath\'{e}odory spaces and Carnot
groups. Next in Subsection \ref{sec2.3} we state some necessary results
concerning the variable Muckenhoupt class $A_{p(\cdot ),q(\cdot )}$ and its
properties, and we also give conditions for the weighted norm inequalities
for the Hardy-Littlewood maximal operator and fractional integral operators
on homogeneous spaces. In Subsection \ref{sec2.2} we give the definition and
properties of the weighted Sobolev spaces with variable exponent $W_{\omega
}^{1,p(\cdot )}(\Omega )$ and $W_{\omega ,0}^{1,p(\cdot )}(\Omega )$. In
Section \ref{sec3} we prove our main results about variable weighted Poincar%
\'{e} inequalities: Theorems \ref{prin}, \ref{Poincare}, \ref{prin2} and \ref%
{Poincare2}. Finally, in Section \ref{sec4} we deal with the elliptic
problem \eqref{1} and we prove Theorems \ref{carac} and \ref{ext-unc}.

For the rest of the paper $c$ and $C$ will denote constants that may vary
line by line. Also, $C(\omega ,p(\cdot ))$ denotes that a constant $C$
depends on the weight $\omega $ and the exponent $p(\cdot )$, for example.

\section{Preliminaries}

\label{sec2}

\subsection{Analysis on Carnot-Carath\'eodory spaces}

\label{sec2.1}

\ 

Let $\mathcal{X}$ be a nonempty set. We say $(\mathcal{X},d,\mu )$ is a
homogeneous space if:

\begin{itemize}
\item $\mathcal{X}$ is endowed with a quasi-metric $d$ such that the balls
are open in the topology induced by $d$ (in particular, they form a base).
Recall that $d$ is a quasi-metric if:

\begin{enumerate}
\item $d(x,y)=0$ if only if $x=y$,

\item $d(x,y)=d(y,x)$ for all $x$, and $y$ in $\mathcal{X}$,

\item 
\begin{equation}
d(x,y)\leq K(d(x,z)+d(z,y)),\qquad \text{for all }\,x,y,z\in \mathcal{X},
\label{k}
\end{equation}%
where $K\geq 1$ is independent of $x$, $y$ and $z$.
\end{enumerate}

\item $\mu$ is a positive borel measure on $\mathcal{X}$, satisfying the
doubling condition 
\begin{equation*}
\mu(B(x,2r)) \leq C \mu(B(x,r)), \ x\in X, r>0;
\end{equation*}
where $B(x,r)=\{y \in \mathcal{X}:d(x,y)<r\}$ is the ball of radius $r$ and
centered in $x$.
\end{itemize}

Let $X=(X_{1},\cdots ,X_{n_{1}})$ be a family of infinitely differentiable
vector fields defined in $\mathbb{R}^{n}$. We identify each $X_{j}$ with the
first order differential operator acting on Lipschitz functions. For a
Lipschitz function $f$ we define $Xf(x)=(X_{1}f(x),\dots ,X_{n_{1}}f(x))$, 
\begin{equation*}
|Xf(x)|=\left( \sum_{j=1}^{n_{1}}|X_{j}f(x)|^{2}\right) ^{1/2},
\end{equation*}%
and for $f=(f_{1},\cdots ,f_{n_{1}})$, 
\begin{equation*}
\text{div}(f)(x)=\sum_{j=1}^{n_{1}}X_{j}f_{j}(x).
\end{equation*}%
Given an open, connected set $\Omega $, $X$ is said to satisfy H\"{o}%
rmander's condition in $\Omega $ if there exists a neighborhood $\Omega _{0}$
of $\Omega $ and $m\in \mathbb{N}$ such that the family of commutators of
the vector fields in $X$ up to length $m$ span $\mathbb{R}^{n}$ at every
point of $\Omega _{0}$. Hereafter, we will assume that $X$ satisfies H\"{o}%
rmander's condition on every bounded, connected subset of $\mathbb{R}^{n}$.

Let $C_X$ be the family of absolutely continuous curves $\gamma : [a, b] \to%
\mathbb{R}^n$, $a\leq b$, such that there exist measurable functions $c_j(t)$%
, $a \leq t \leq b$, $j = 1,\cdots ,n_1$, satisfying $%
\sum_{j=1}^{n_1}c_j(t)^2\leq 1$ and $\gamma^{\prime
}(t)=\sum_{j=1}^{n_1}c_j(t)X(\gamma(t))$ for almost every $t \in [a, b]$.
Given $x, y \in \Omega$, define 
\begin{equation*}
d(x,y)=\inf\{T>0: \exists \gamma \in C_X \text{ s. t. } \gamma(0)=x \text{
and } \gamma(T)=y\}.
\end{equation*}
The function $d$ is a quasi-metric on $\Omega$ called the
Carnot-Carath\'eodory metric associated to $X$, and the pair $(\Omega, d)$
is said to be a Carnot-Carath\'eodory space. We refer the reader to \cite%
{NSW} for more details on Carnot-Carath\'eodory spaces.

Nagel, Stein and Wainger proved in \cite{NSW} that for every compact set $%
K\subset \Omega $ there exist positive constants $R$ and $C(n_{1})$ such
that 
\begin{equation*}
|B(x,2r)|\leq C(n_{1})|B(x,r)|,\,\forall \,x\in K\text{ and }0<r<R,
\end{equation*}%
where $|E|$ denotes the Lebesgue measure of the measurable set $E$.
Therefore the Carnot-Carath\'{e}odory space $(\Omega ,d,dx)$ is a type
homogeneous space. The homogeneous dimension of the Carnot-Carath\'{e}odory
space is $Q=\log _{2}C(n_{1})$.

An important type of Carnot-Carath\'{e}odory space are the Carnot groups. A
Carnot group is a simply connected and connected Lie group $\mathbb{G}$,
whose Lie algebra $\mathfrak{g}$ is stratified, this means that $\mathfrak{g}
$ admits a vector space decomposition $\mathfrak{g}=V_{1}\oplus \cdots
\oplus V_{m}$ with grading $[V_{1},V_{j}]=V_{j+1}$, for $j=1,\cdots ,m-1,$
and has a family of dilations $\{\delta _{\epsilon }\}_{\epsilon >0}$ such
that $\delta _{\epsilon }X=\epsilon ^{j}X$ if $X\in V_{j}$. Denote by $%
n=n_{1}+\cdots +n_{m}$, where $n_{k}$ is the dimension of $V_{k}$. In this
case, the homogeneous dimension is $Q=\sum_{k=1}^{m}kn_{k}$, see \cite{FS}.

It is well-known that the Haar measure for $\mathbb{G}$ is a Lebesgue
measure. Let $X=\{X_{1},\dots ,X_{n_{1}}\}$ be a basis of $V_{1}$. It is
clear that $X$ satisfies H\"{o}rmander's condition in $\mathbb{R}^{n}$ and
is naturally associated with $\{X_{j}\}_{j=1}^{n_{1}}$ a Carnot-Carath\'{e}%
odory metric $d(x,y)$ for $x,y\in \mathbb{G}$. The geometry of the metric
space $(\mathbb{G},d)$ is described in \cite{NSW}, \cite{FP} and \cite{S}.
In particular, the $d-$topology and the Euclidean topology are equivalent
and the Carnot group $\mathbb{G}$ is a homogeneous space with the metric $d$
and the Lebesgue measure. We get $|B(x,r)|\sim r^{Q}$ and 
\begin{equation*}
|B(x,2r)|\lesssim 2^{Q}|B(x,r)|.
\end{equation*}

Let $\{X_{1},\cdots ,X_{n}\}$ be a basis of the Lie algebra $\mathfrak{g}$
and $I=(i_{1},\dots ,i_{n})\in \mathbb{N}_{0}^{n}$ a multiindex. We set $%
X^{I}=X_{1}^{i_{1}}X_{2}^{i_{2}}\dots X_{n}^{i_{n}}$. The operators $X^{I}$
form a basis for the algebra of left invariant differential operators on $%
\mathbb{G}$, by the Poincar\'{e}-Birkhoff-Witt theorem. The order of the
differential operators $X^{I}$ is $|I|=i_{1}+i_{2}+\dots +i_{n}$ and its 
\textit{homogeneous degree} is $d(I)=\lambda _{1}i_{1}+\lambda
_{2}i_{2}+\dots +\lambda _{n}i_{n}$, where $\lambda _{j}=i$ if $X_{j}\in
V_{i}$. For $m\in \mathbb{N}$, we define 
\begin{equation*}
|X^{m}f|=\left( \sum_{I:d(I)=m}|X^{I}f|^{2}\right) ^{1/2}.
\end{equation*}

A function $f$ on $\mathbb{G}\backslash\{0\}$ will be called \textit{%
homogeneous of degree $m$} if $f\circ\delta_{r}=r^{m}f$ for $r>0$. We say
that a function $P$ on $\mathbb{G}$ is a \textit{polynomial} if $P\circ\exp$
is a polynomial on $\mathfrak{g}$. Let $\xi_{1},\dots,\xi_{n}$ be the basis
for the linear forms on $\mathfrak{g}$ dual to the basis $X_{1},\dots,X_{n}$
on $\mathfrak{g}$. Let us set $\eta_{j}=\xi_{j}\circ\exp^{-1}$, then $%
\eta_{1},\dots,\eta_{n}$ are polynomials on $\mathbb{G}$ which form a global
coordinate system on $\mathbb{G}$, and generate the algebra of polynomials
on $\mathbb{G}$. Thus, every polynomial on $\mathbb{G}$ can be written
uniquely as $P=\sum\limits_{I}a_{I}\eta^{I}$, for $\eta^{I}=\eta_{1}^{i_{1}}%
\dots\eta_{n}^{i_{n}}$, $a_{I}\in\mathbb{C}$ where all but finitely many of
them vanish. The function $\eta^{I}$ is homogeneous of order $d(I)$. The 
\textit{homogeneous order} is $\max\{d(I):a_{I}\neq 0\}$. For each $j\in%
\mathbb{N}$ we define the space $\mathcal{P}_{j}$ of polynomials of
homogeneous degree $\le j$. See \cite{FS}.

Finally we define the weak Boman chain condition for a domain $\Omega$ in a
metric space $\mathcal{X}$.

\begin{definition}
A domain $\Omega$ in a metric space $\mathcal{X%
}$ is said to satisfy the weak Boman chain condition of type $\sigma$, $%
\Lambda$, or to be a member of $\mathcal{F}(\sigma,\Lambda)$, if there exist
constants $\sigma>1$, $\Lambda > 1$, and a family $\mathcal{F}$ of metric
balls $B \subset \Omega$ such that

\begin{itemize}
\item $\Omega=\bigcup_{B\in\mathcal{F}}B$;

\item $\Sigma_{B\in\mathcal{F}}\chi_{\sigma B}(x)\leq \Lambda \chi_\Omega(x)$
for all $x\in\mathcal{X}$;

\item there exists a central ball $B_{0}\in \mathcal{F}$ such that, for each
ball $B\in \mathcal{F}$, there exist a positive integer $k=k(B)$ and a chain 
$\{B_{j}\}_{j=0}^{k}$ of balls in $\mathcal{F}$ for which $B_{k}=B$ and each 
$B_{j}\cap B_{j+1}$ contain a ball $D_{j}\in \mathcal{F}$ with $B_{j}\cup
B_{j+1}\subset \Lambda D_{j}$;

\item $B\subset \Lambda B_j$ for $j=1,\cdots, k(B)$.
\end{itemize}
\end{definition}

\subsection{Weighted variable Lebesgue space and clasical operators}

\label{sec2.3}

\ 

The Hardy-Littlewood maximal operator and the fractional integral operator
of order $\alpha $ on homogeneous spaces are defined respectively by 
\begin{equation*}
Mf(x)=\sup_{B\ni x}\dfrac{1}{\mu (B)}\int_{B}f(y)d\mu (y),
\end{equation*}%
\begin{equation*}
I_{\alpha }f(x)=\int_{\Omega }f(y)\dfrac{d(x,y)^{\alpha }}{\mu (B(x,d(x,y)))}%
d\mu (y).
\end{equation*}%
The weighted boundedness of these operators has been extensively studied in
the literature, see \cite{CUC}, \cite{FLW}, \cite{LW}, \cite{LLT}, \cite{SW}%
. We will need to prove bounds where the weight is in the variable
Muckenhoupt class $A_{p(\cdot ),q(\cdot )}$ in homogeneous spaces $(\mathcal{%
X},d,\mu )$.

%DESDE ACA

We must first start by introducing important definitions about variable
exponents.

Let $(\mathcal{X},d,\mu )$ be a homogeneous space and let $\mathcal{P}(%
\mathcal{X})$ be the collection of all measurable functions $p(\cdot ):%
\mathcal{X}\rightarrow \lbrack 1,\infty ]$. Given a set $E\subset \mathcal{X}
$, $\mu -measurable$, we define%
\begin{eqnarray*}
p_{E}^{-} &=&ess~\inf_{x\in E}~p(x)\text{,} \\
p_{E}^{+} &=&ess~\sup_{x\in E}~p(x)\text{.}
\end{eqnarray*}%
If $E=\mathcal{X}$ we write $p^{-}$ and $p^{+}$.

We next define the jump condition of a variable exponent.

\begin{definition}
Let $\Omega $ be a subset of a Carnot-Carath\'{e}odory space $(\mathcal{X}%
,d,dx)$ and let $p(\cdot )\in \mathcal{P}(\mathcal{X})$. If $p_{\Omega
}^{+}<\infty $ and if there exists $\delta >0$ such that for every $x\in
\Omega $ either 
\begin{equation*}
p_{B(x,\delta )}^{-}\geq Q
\end{equation*}%
or 
\begin{equation*}
p_{B(x,\delta )}^{+}\leq \dfrac{Qp_{B(x,\delta )}^{-}}{Q-p_{B(x,\delta )}^{-}%
}
\end{equation*}%
holds, then $p(\cdot )$ is said to satisfy the $jump\ condition$ in $\Omega $
with constant $\delta $.
\end{definition}

\begin{remark}
The hypothesis $p^{+}<Q$ implies that $p_{B(x,\delta )}^{+}\leq \dfrac{%
Qp_{B(x,\delta )}^{-}}{Q-p_{B(x,\delta )}^{-}}$.
\end{remark}

Let $p(\cdot )\in \mathcal{P}(\mathcal{X})$, $\Omega \subset \mathcal{X}$ be
an open set and $\mu $ be a borel measure. We define the $variable\
Lebesgue\ space$ $L_{\mu }^{p(\cdot )}(\Omega )$ as the set of measurable
functions $f$ on $\Omega $ for which the modular 
\begin{equation*}
\rho _{p(\cdot ),\mu }(f)=\int_{\Omega \diagdown \Omega _{\infty }}\lvert
f(x)\rvert ^{p(x)}d\mu +\left\Vert f\right\Vert _{L_{\mu }^{\infty }(\Omega
_{\infty })}\text{,}
\end{equation*}%
satisfies $\rho _{p(\cdot ),\mu }(f/\lambda )<\infty $ for some $\lambda >0$
and $\Omega _{\infty }=\left\{ x\in \Omega :p(x)=\infty \right\} $.

When the exponent and the measure are clear from context, we simply write $%
\rho _{p(\cdot ),\mu }=\rho $.

If $1\leq p^{-}\leq p^{+}<\infty $ and $\mu $ is a borel regular measure, it
is well known that the variable Lebesgue space $L_{\mu }^{p(\cdot )}(\Omega
) $ is a Banach space with the norm 
\begin{equation*}
\lVert f\rVert _{L_{\mu }^{p(\cdot )}(\Omega )}=\inf \left\{ \lambda >0:\rho
(f/\lambda )\leq 1\right\} ;
\end{equation*}%
if $1<p^{-}\leq p^{+}<\infty $, $L_{\mu }^{p(\cdot )}(\Omega )$ is also a
reflexive space.

When the measura is given by $\omega (x)^{p(x)}\,dx$, with $\omega >0$ a
measurable function, we write $\lVert f\rVert _{L_{\omega }^{p(\cdot
)}(\Omega )}$ in place of $\lVert f\rVert _{L_{\omega ^{p(\cdot )}}^{p(\cdot
)}(\Omega )}$.

If $1\leq p(\cdot )\leq q(\cdot )$ and $\Omega $ has finite measure then
there exists a constant $C(\Omega )$ such that for every $f\in L_{\omega
}^{q(\cdot )}(\Omega )$ 
\begin{equation*}
\Vert f\Vert _{L_{\omega }^{p(\cdot )}(\Omega )}\leq C(\Omega )\Vert f\Vert
_{L_{\omega }^{q(\cdot )}(\Omega )}.
\end{equation*}

We will need the following properties proved in \cite{C-UFN}.

%\begin{lemma}
%\label{norma-modular} Let $p$ a variable exponent, then given any domain $%
%\Omega$,

%\begin{itemize}
%\item if $\|u\|_{L^{p(\cdot)}(\Omega)} \leq 1,$ \,\,$\|u\|_{L^{p(\cdot)}(%
%\Omega)}^{p_+(\Omega)}\leq \int_{\Omega}|u(x)|^{p(x)}\,dx \leq
%\|u\|_{L^{p(\cdot)}(\Omega)}^{p_-(\Omega)}$,

%\item if $\|u\|_{L^{p(\cdot)}(\Omega)} \geq 1,$ \,\,$\|u\|_{L^{p(\cdot)}(%
%\Omega)}^{p_-(\Omega)}\leq \int_{\Omega}|u(x)|^{p(x)}\,dx \leq
%\|u\|_{L^{p(\cdot)}(\Omega)}^{p_+(\Omega)}$.
%\end{itemize}

%In particular, if $\|u\|_{L^{p(\cdot)}(\mathbb{G})}\leq 1$, 
%\begin{equation*}
%\int_{\mathbb{G}}|u(x)|^{p(x)}\,dx \leq \|u\|_{L^{p(\cdot)}(\mathbb{G})}.
%\end{equation*}
%\end{lemma}

\begin{lemma}
\label{dilationExpo} Let $(\mathcal{X},\mu )$ be a measurable space with $%
\mu $ a borel regular measure. Let $p(\cdot )\in \mathcal{P}(\mathcal{X})$
such that $|\mathcal{X}_{\infty }|=0$. Then for all $s$ with $1/p^{-}\leq
s<\infty $, 
\begin{equation*}
\Vert |f|^{s}\Vert _{L^{p}(\cdot )(\mathcal{X})}=\Vert f\Vert _{L^{sp(\cdot
)}(\mathcal{X})}^{s}.
\end{equation*}
\end{lemma}

\begin{lemma}
\label{norma-modular} Let $(\mathcal{X},\mu )$ be a measurable space with $%
\mu $ a borel regular measure. Let $p(\cdot )\in \mathcal{P}(\mathcal{X})$
with $1\leq p^{-}\leq p^{+}<\infty $. Then, for any domain $\Omega $,

\begin{itemize}
\item if $\|f\|_{L^{p(\cdot)}_\mu(\Omega)} \leq 1,$ \,\,$\|f\|_{L^{p(%
\cdot)}_\mu(\Omega)}^{p^+_\Omega}\leq \int_{\Omega}|f(x)|^{p(x)}\,d\mu(x)
\leq \|f\|_{L^{p(\cdot)}_\mu(\Omega)}^{p^-_\Omega}$,

\item if $\|f\|_{L^{p(\cdot)}_\mu(\Omega)} \geq 1,$ \,\,$\|f\|_{L^{p(%
\cdot)}_\mu(\Omega)}^{p^-_\Omega}\leq \int_{\Omega}|f(x)|^{p(x)}\,d\mu(x)
\leq \|f\|_{L^{p(\cdot)}_\mu(\Omega)}^{p^+_\Omega}$.
\end{itemize}

%In particular, if $\|f\|_{L^{p(\cdot)}_\mu(\mathcal{X})}\leq 1$, 
%\begin{equation*}
%\int_{\mathcal{X}}|f(x)|^{p(x)}\,d\mu(x) \leq \|f\|_{L^{p(\cdot)}_\mu(%
%\mathcal{X})}.
%\end{equation*}
\end{lemma}

\begin{remark}
\label{modular-norma} The last lemma implies

\begin{itemize}
\item $\|f\|_{L^{p(\cdot)}_\mu}=1$ if and only if $\rho_{p(\cdot)}(f)=1$.

\item if $\rho_{p(\cdot)}(f)\leq C$ then $\|f\|_{L^{p(\cdot)}_\mu}\leq
\max\{ C^{1/p_-}, C^{1/p_+} \}$.

\item if $\|f\|_{L^{p(\cdot)}_\mu}\leq C$ then $\rho_{p(\cdot)}(f)\leq
\max\{ C^{p_-}, C^{p_+} \}$.
\end{itemize}
\end{remark}

\begin{lemma}[H\"{o}lder's inequality]
\label{HI} Let $(\mathcal{X},\mu )$ be a measurable space with $\mu $ a
borel regular measure. Let $p(\cdot ),\,q(\cdot ),\,r(\cdot )\,\in \mathcal{P%
}(\mathcal{X})$ with $\frac{1}{r(x)}=\frac{1}{q(x)}+\frac{1}{p(x)}$. If $%
f\in L_{\mu }^{p(\cdot )}(\mathcal{X})$ and $g\in L_{\mu }^{q(\cdot )}(%
\mathcal{X})$ then $fg\in L_{\mu }^{r(\cdot )}(\mathcal{X})$ and there
exists a constant $c$ such that 
\begin{equation*}
\Vert fg\Vert _{L_{\mu }^{r(\cdot )}(\mathcal{X})}\leq c\Vert f\Vert
_{L_{\mu }^{p(\cdot )}(\mathcal{X})}\Vert g\Vert _{L_{\mu }^{q(\cdot )}(%
\mathcal{X})}.
\end{equation*}
\end{lemma}

Given $p(\cdot )$, the conjugate exponent $p^{\prime }(\cdot )$ is defined by%
\begin{equation*}
\frac{1}{p(x)}+\frac{1}{p^{\prime }(x)}=1\text{.}
\end{equation*}

%HASTA ACA

\begin{definition}
Let $(\mathcal{X},d,\mu )$ be a homogeneous space and $p(\cdot ),q(\cdot
)\in \mathcal{P}(\mathcal{X})$ be such that, $1<q(\cdot )<\infty $ for some $%
\gamma ,~0\leq \gamma <1$, 
\begin{equation*}
\frac{1}{p(x)}-\frac{1}{q(x)}=\gamma \text{.}
\end{equation*}%
Given $\omega $ a locally integrable function such that $0<\omega (x)<\infty 
$ a.e. $x$, we say that $\omega $ belongs to the Muckenhoupt class $%
A_{p(\cdot ),q(\cdot )}$ if 
\begin{equation*}
\lbrack \omega ]_{A_{p(\cdot ),q(\cdot )}}=\sup_{B}\left\vert B\right\vert
^{\gamma -1}\left\Vert \omega \chi _{B}\right\Vert _{q(\cdot )}\left\Vert
\omega ^{-1}\chi _{B}\right\Vert _{p^{\prime }(\cdot )}<\infty \text{,}
\end{equation*}%
where the supremum is taken over all balls $B\subset \mathcal{X}$.
\end{definition}

The above definiton has two inmediate consequences. First, $\omega \in L_{%
\text{loc}}^{q(\cdot )}$ and $\omega ^{-1}\in L_{\text{loc}}^{p^{\prime
}(\cdot )}$. We note that since $p(\cdot )<q(\cdot )$ implies that $\omega
\in L_{\text{loc}}^{p(\cdot )}$ and also since $q^{\prime }(\cdot
)<p^{\prime }(\cdot )$ then $\omega ^{-1}\in L_{\text{loc}}^{q^{\prime
}(\cdot )}$. Second, if $1<p^{-}$ note that $\omega \in A_{p(\cdot ),q(\cdot
)}$ if and only if $\omega ^{-1}\in A_{q^{\prime}(\cdot ), p^{\prime }(\cdot )}$.

If $\gamma =0$ we write $A_{p(\cdot )}$ in place of $A_{p(\cdot ),p(\cdot )}$%
.

\begin{remark}
If we consider a domain $\Omega $ we say that the weight $\omega $ belongs
to the Muckenhoupt class $A_{p(\cdot ),q(\cdot )}(\Omega )$ if the
definition is satisfied by taking supreme over the balls $B\subset \Omega $.
For some results in this paper this definition is enough, but to simplify
the presentation we have always considered the definition of the weights in
all $\mathcal{X}$.
\end{remark}

We will use the following properties of the Muckenhoupt class

\begin{lemma}
\label{propA} If $p$ and $q$ are constant we get

\begin{itemize}
\item $p<q$ implies $A_{p}\subset A_{q}$ and we define $A_{\infty
}=\bigcup_{p\geq 1}A_{p}$.

\item If $p>1$ then $A_{p,q}\subset A_{1+q/p^{\prime }}$.

\item If $\omega \in A_p$ then $w\in L^1_{\text{loc}}$.

\item If $\omega_1 , \omega_2 \in A_1$ then $\omega=\omega_1^{1-p}\omega_2
\in A_p$.

\item If $\omega \in A_{p,q}$ then $\omega^q$ is doubling.
\end{itemize}
\end{lemma}

\begin{definition}
We say that an exponent $p(\cdot )\in \mathcal{P}(\mathcal{X})$ is \textit{%
locally\ Log-H\"{o}lder\ continuous}, $p(\cdot )\in LH_{0}$, if there exists
a constant $C_{0}$ such that for any $x,y\in \mathcal{X}$, with $d(x,y)<%
\frac{1}{2}$, 
\begin{equation*}
\lvert p(x)-p(y)\rvert <\dfrac{-C_{0}}{\log (d(x,y))}.
\end{equation*}%
We say that $p(\cdot )\in \mathcal{P}(\mathcal{X})$ is \textit{Log-H\"{o}%
lder\ continuous\ at\ infinity,} $p(\cdot )\in LH_{\infty }$, with respect
to a base point $x_{0}\in \mathcal{X}$ if there exist constants $C_{\infty }$
and $p_{\infty }$ such that for every $x\in \mathcal{X}$, 
\begin{equation*}
\lvert p(x)-p_{\infty }\rvert <\dfrac{C_{\infty }}{\log (e+d(x,x_{0}))}.
\end{equation*}%
If $p(\cdot )\in LH=LH_{0}\cap LH_{\infty }$ we say that $p(\cdot )$ is 
\textit{globally\ Log-H\"{o}lder\ continuous}.
\end{definition}

In \cite{CUC} the authors prove the following result.

\begin{theorem}
Let $p(\cdot )\in \mathcal{P}(\mathcal{X})$, $1<p_{-}\leq p_{+}<\infty $,
and suppose $p(\cdot )\in LH$. Then for every $\omega \in A_{p(\cdot )}$, 
\begin{equation}
\lVert Mf\rVert _{L_{\omega }^{p(\cdot )}}\leq C\lVert f\rVert _{L_{\omega
}^{p(\cdot )}}.  \label{BoundedMaximal}
\end{equation}%
Conversely, given any $p(\cdot )$ and $\omega $, if (\ref{BoundedMaximal})
holds for $f\in L_{\omega }^{p(\cdot )}$, then $p_{-}>1$ and $\omega \in
A_{p(\cdot )}$.
\end{theorem}

\bigskip Given $p(\cdot )\in \mathcal{P}(\mathcal{X})$ and a weight $\omega $%
, we will say that $(p(\cdot ),\omega )$ is a $M-pair$ if the maximal
operator $M$ is bounded on $L_{\omega }^{p(\cdot )}$ and $L_{\omega
^{-1}}^{p^{\prime }(\cdot )}$. We observe that if $p(\cdot )\in LH$, $%
1<p^{-}\leq p^{+}<\infty $ and $\omega \in A_{p(\cdot )}$ then $(p(\cdot
),\omega )$ is a $M-pair$. It is a sufficient condition but it is not known
if it is also a necessary condition.

\subsection{Weighted variable Sobolev spaces}

\label{sec2.2}

\ 

In this section we are going to define Sobolev spaces with variable exponent
and we establish their basic properties.

Let $\Omega $ be a domain in a Carnot group $\mathbb{G}$, let $m$ be a
positive integer and $1<p^{-}\leq p^{+}<\infty $. The Sobolev space $%
W_{\omega }^{m,p(\cdot )}(\Omega )$ associated with the vector fields $%
\{X_{1},\cdots ,X_{n_{1}}\}$ consist of all functions $f\in L_{\omega
}^{p(\cdot )}(\Omega )$ with the absolute value of distributional
derivatives $|X^{I}f|\in L_{\omega }^{p(\cdot )}(\Omega )$ for every $X^{I}$
differential operators with homogeneous degree $d(I)\leq m$.

%For $X_j$ will denote by $X_j^*$ the infinitely differentiable vector fields such that
%$$
%\int_{\Xx}X_j(f)g\,dx=-\int_{\Xx}fX_j^*(g)\,dx,
%$$ 
%for all $f,\,\,g \in C_0^\infty(\Xx)$. If $\Xx$ is a Carnot group we have $X_j^*=X_j$. 
Here we say that the distributional derivative $X^If$ exists and equals a
locally integrable function $g_I$ if for every $\phi \in C_0^\infty(\Omega)$%
, 
\begin{equation*}
\int_\Omega fX^{I}\phi\,dx=(-1)^{d(I)}\int_\Omega g_I\phi\,dx.
\end{equation*}
$W^{m,p(\cdot)}_\omega(\Omega)$ is equipped with the norm 
\begin{equation*}
\|f\|_{W^{m,p(\cdot)}_\omega(\Omega)}=\|f\|_{L^{p(\cdot)}_\omega(\Omega)}+%
\sum_{I:1\leq d(I)\leq m}\|X^If\|_{L^{p(\cdot)}_\omega(\Omega)}.
\end{equation*}

\begin{definition}
Let $\Omega $ be a open set in a Carnot group $\mathbb{G}$ and $p(\cdot )\in 
\mathcal{P}(\mathbb{G})$ with $1<p^{-}\leq p^{+}<\infty $. Given a weight $%
\omega >0$, we define the Sobolev space $W_{\omega ,0}^{1,p(\cdot )}(\Omega
) $ by $\overline{C_{0}^{1}(\Omega )}$, where the closure is with respect to
the norm of $W_{\omega }^{1,p(\cdot )}(\Omega )$.
\end{definition}

%The proof the next theorem is standard and omitted.
%\begin{theorem}\label{reflex}
%The space $W^{1,p(\cdot)}_{\omega,0}(\Omega)$ is a reflexive Banach and $\left( W_{\omega }^{1,p(\cdot )}(\Omega )\right) ^{\ast}\cong W_{\omega ^{-1}}^{1,p^{\prime }(\cdot )}(\Omega )$.
%\end{theorem}

\begin{remark}
When the exponent $p$ is constant we get that $C^{m}(\Omega )\cap
W^{m,p}(\Omega )$ is dense in $W^{m,p}(\Omega )$, this is not true in the
general case, see \cite{ER} and \cite{ER-1}. The variable exponent $p(\cdot
)\in \mathcal{P}(\mathbb{G})$ is said to satisfy the density condition if $%
C^{m}(\Omega )\cap W^{m,p(\cdot )}(\Omega )$ is dense in $W^{m,p(\cdot
)}(\Omega )$. It is not known in general when the density condition holds.
Edmund and R\'{a}kosn\'{\i}v, in \cite{ER-1}, consider $\Omega =\mathbb{R}%
^{n}$ and they have proved that the following condition is sufficient: for
every $x\in \mathbb{R}^{n}$ there exists a number $h(x)>0$ and a vector $\xi
(x)\in \mathbb{R}^{n}\diagdown \{0\}$ such that

\begin{itemize}
\item $h(x)<|\xi(x)|\leq 1,$ and

\item $p(x)\leq p(x+y)$ for a.e. $x\in\mathbb{R}^n$ and $y\in
\bigcup_{0<t\leq 1}B_{t\xi(x),th(x)}$.
\end{itemize}
\end{remark}

%\section{The Sobolev space $W^{1,p(\cdot)}_\omega(\Omega)$}\label{sec4.1}

In \cite{CF} the authors prove that the space $L^{p(\cdot )}(\Omega )$ is a
reflexive Banach space for $1<p^{-}\leq p^{+}<\infty $, observe that if $%
\{u_{k}\}$ is a Cauchy sequence in $L_{\omega }^{p(\cdot )}(\Omega )$ then $%
\{u_{k}\omega \}$ is a Cauchy sequence in $L^{p(\cdot )}(\Omega )$ and it is
easy to see that $L_{\omega }^{p(\cdot )}(\Omega )$ is a reflexive Banach
space. We are going to study these properties for the space $W_{\omega
}^{1,p(\cdot )}(\Omega )$. First, we observe that if $\{u_{k}\}$ is a Cauchy
sequence in $W_{\omega }^{1,p(\cdot )}(\Omega )$ we have that $\{u_{k}\}$, $%
\{X_{1}u_{k}\}$, $\cdots ,$ $\{X_{n_{1}}u_{k}\}$ are Cauchy sequences in $%
L_{\omega }^{p(\cdot )}(\Omega )$. We are going to show that $W_{\omega
}^{1,p(\cdot )}(\Omega )$ is closed in $L_{\omega }^{p(\cdot )}(\Omega
)\times \cdots \times L_{\omega }^{p(\cdot )}(\Omega )$. Since $L_{\omega
}^{p(\cdot )}(\Omega )$ is a Banach space exist $(u,U):=(u,U_{1},\cdots
,U_{n_{1}})$ limits of $\{u_{k}\}$, $\{X_{1}u_{k}\}$, $\cdots ,$ $%
\{X_{n_{1}}u_{k}\}$ in $L_{\omega }^{p(\cdot )}(\Omega )$ respectively.

In \cite{FKS}, p. 91, for $p=2$ the authors construct a weight $\omega $ and
a sequence of functions $\{u_{k}\}$ in the interval $I=[0,1]$ such that

\begin{itemize}
\item $\lim_{k\to \infty} \int_I|u_k(x)|^2\omega(x)^2dx=0,$

\item $\lim_{k\to \infty} \int_I|u_k^{\prime}(x)-1|^{2}\omega(x)^2dx=0,$
\end{itemize}

that is, the sequence $\{u_k\}$ converges to $0$ in $L^2_\omega(I)$ and $%
\{u_k^{\prime }\}$ converges to $1$ in $L^2_\omega(I)$. The weight defined
in \cite{FKS} does not belong to $A_\infty$.

Now we are going to show that if $\omega \in A_{p(\cdot ),q(\cdot )}$ then
this cannot happen. Suppose to the contrary that $(u,U)$ and $(u,V)$ are
limit of $\{u_{k}\}$. This is equivalent to say that there exists a Cauchy
sequence $\{u_{k}\}$ in $W_{\omega }^{1,p(\cdot )}(\Omega )$ such that $%
u_{k} $ converges to $0$ in $L_{\omega }^{p(\cdot )}(\Omega )$ and $%
\{|Xu_{k}|\}$ converges to $|U|$ in $L_{\omega }^{p(\cdot )}(\Omega )$ with $%
U\neq 0$. Let $B\subset \Omega $ be any ball and let $g\in C_{0}^{\infty
}(B) $. Then, for $j=1,\cdots ,n_{1}$, by H\"{o}lder inequality, Lemma \ref%
{HI} 
\begin{align*}
\left\vert \int_{B}g(x)U_{j}(x)dx\right\vert & \leq \left\vert
\int_{B}g(x)(U_{j}(x)-X_{j}u_{k}(x))dx\right\vert +\left\vert
\int_{B}g(x)X_{j}u_{k}(x)dx\right\vert \\
& \leq C\Vert U_{j}-X_{j}u_{k}\Vert _{L_{\omega }^{p(\cdot )}(B)}\Vert
g\Vert _{L_{\omega ^{-1}}^{p^{\prime }(\cdot )}(B)}+C\Vert u_{k}\Vert
_{L_{\omega }^{p(\cdot )}(B)}\Vert X_{j}g\Vert _{L_{\omega ^{-1}}^{p^{\prime
}(\cdot )}(B)}.
\end{align*}%
Since $w\in A_{p(\cdot ),q(\cdot )}$ then $\omega ^{-1}\in L_{\text{loc}%
}^{p^{\prime }(\cdot )}$, and as $g,X_{j}g\in L^{\infty }(B)$ we get that $%
\Vert g\Vert _{L_{\omega ^{-1}}^{p^{\prime }(\cdot )}(B)}$ and $\Vert
X_{j}g\Vert _{L_{\omega ^{-1}}^{p^{\prime }(\cdot )}(B)}$ are finite. So if
we take limit when $k\rightarrow \infty $ we obtain that the right side of
the above inequality goes to $0$. Since this holds for all such $B$ and $g$, 
$U_{j}=0$ for all $j=1,\cdots ,n_{1}$.

Now, we prove that if $\{u_{k}\}$ converges to $(u,U)=(u,U_{1},\cdots
,U_{n_{1}})$ then $U_{j}$ corresponds to the distributional $X_{j}$
derivative of $u$, that is, 
\begin{equation}
\int_{\Omega }U_{j}\phi \,dx=-\int_{\Omega }uX_{j}\phi \,dx,\qquad \text{for
all }\,\phi \in C_{0}^{\infty }(\Omega ).  \label{Uj}
\end{equation}%
We first note that both of the integrals appearing in \eqref{Uj} exist:
after multiplying by $\omega (x)\omega (x)^{-1}$, using H\"{o}lder
inequality, Lemma \ref{HI}, and since $u,U_{j}\in L_{\omega }^{p(\cdot
)}(\Omega )$, and $\phi ,X_{j}\phi \in L_{\omega ^{-1}}^{p^{\prime }(\cdot
)}(\Omega )$, we have that, 
%As $u_k\to u$ and $X_ju_k \to U_j$ in $L^{p(\cdot)}_\omega(\Omega)$ there exist $K$ such that $k>K$ we have $\|u_k-u\|_{L^{p(\cdot)}_\omega(\Omega)}<1$ and $\|X_ju_k-U_j\|_{L^{p(\cdot)}_\omega(\Omega)}<1$, by  Lema \ref{norma-modular}
%\begin{align*}
%\lim_{k\to \infty} \int_{\Omega}|u_k(x)-u(x)|^{p(x)}\omega(x)^{p(x)}dx&\leq \lim_{k\to \infty}\|u_k-u\|_{L^{p(\cdot)}_\omega(\Omega)}=0,\\
%\lim_{k\to \infty} \int_{\Omega}|X_ju_k(x)-U_j(x)|^{p(x)}\omega(x)^{p(x)}dx&\leq \lim_{k\to \infty}\|X_ju_k-U_j\|_{L^{p(\cdot)}_\omega(\Omega)}=0.
%\end{align*}

\begin{align*}
\left\vert \int_{\Omega }U_{j}\phi +uX_{j}\phi \,dx\right\vert & \leq
\left\vert \int_{\Omega }(U_{j}-X_{j}u_{k})\phi \,dx\right\vert +\left\vert
\int_{\Omega }(u-u_{k})X_{j}\phi \,dx\right\vert \\
& \leq C\Vert u_{k}-u\Vert _{L_{\omega }^{p(\cdot )}(\Omega )}\Vert \phi
\Vert _{L_{\omega ^{-1}}^{p^{\prime }(\cdot )}(\Omega )} \\
& \qquad \quad +\Vert X_{j}u_{k}-U_{j}\Vert _{L_{\omega }^{p(\cdot )}(\Omega
)}\Vert X_{j}\phi \Vert _{L_{\omega ^{-1}}^{p^{\prime }(\cdot )}(\Omega )},
\end{align*}%
and so as $k\rightarrow \infty $ we obtain \eqref{Uj}.

We have shown that $W^{1,p(\cdot)}_\omega(\Omega)$ is a closed subspace of $%
L^{p(\cdot)}_\omega(\Omega) \times \cdots \times L^{p(\cdot)}_\omega(\Omega)$%
, therefore $W^{1,p(\cdot)}_\omega(\Omega)$ is a reflexive Banach space for $%
1<p^-\leq p^+<\infty$ and $\omega\in A_{p(\cdot),q(\cdot)}$.

\section{Proof of the main results.}\label{sec3}

In order to prove the theorems we will need to obtain the bounded fractional integral operator $I_\alpha$ with respect to the weights $\omega\in A_{p(\cdot),q(\cdot)}$. For this we use the next extrapolation theorem in variable Lebesgue spaces with  weight. This result is proved in \cite{C-W} for the case $\Xx=\mathbb{R}^n$. But it is not difficult to adapt for general homogeneous spaces.

We will write the extrapolation theorem for pair of functions $(f,g)$ contained in some family $\Ff$. Hereafter, if we write
$$
\|f\|_{\Xx}\leq C\|g\|_{\mathcal{Y}}. \qquad\qquad (f,g)\in\Ff,
$$
where $\Xx$ and $\mathcal{Y}$ are weighted classical or variable Lebesgue spaces, then mean this inequality is true for every pair $(f,g)\in\Ff$ such that the left hand side of this inequality is finite.
\begin{theorem}\label{extrapolation} 
Let be $(\Xx,d,\mu)$ a homogeneous space. Suppose that for some $p_{0},q_{0}$, $1<p_{0}\leq
q_{0}<\infty $, and every $\omega_{0}\in A_{p_{0},q_{0}}$, 
\begin{equation}  \label{conditionbase}
\left( \int_{\Xx}f(x)^{q_{0}}\omega _{0}(x)^{q_{0}}dx\right) ^{ 
\frac{1}{q_{0}}}\leq C\left( \int_{\Xx}g(x)^{p_{0}}\omega
_{0}(x)^{p_{0}}dx\right) ^{\frac{1}{p_{0}}}\text{, \ }(f,g)\in \mathcal{F} 
\text{.}
\end{equation}
Given $p(\cdot ),q(\cdot )\in \mathcal{P}(\Xx)$, suppose 
\begin{equation*}
\frac{1}{p(x)}-\frac{1}{q(x)}=\frac{1}{p_{0}}-\frac{1}{q_{0}}\text{.}
\end{equation*}
Define $\sigma \geq 1$ by $\frac{1}{\sigma ^{\prime }}=\frac{1}{p_{0}}-\frac{
1}{q_{0}}$. If $\omega \in A_{p(\cdot ),q(\cdot )}$ and $\left( \frac{q(\cdot
)}{\sigma },\omega ^{\sigma }\right) $ is a $M-pair$, then 
\begin{equation*}
\left\Vert f\right\Vert_{L^{q(\cdot )}_\omega(\Xx)}\leq C\left\Vert g\right\Vert
_{L^{p(\cdot )}_\omega (\Xx)}\text{, \ \ }(f,g)\in \mathcal{F}\text{ .}
\end{equation*}
The theorem holds if $p_0 =1$ if we assume only that the maximal operator is
bounded on $L^{(q(\cdot)/q_0)^{\prime }}(\omega^{-q_0})$.
\end{theorem}

For the proof we will need the following three propositions. Both are tested in \cite{C-W} for $\RR^n$ with Lebesgue measure. This results holds on spaces of homogeneous type. The first proposition is useful to find the $A_1$ weights. We apply the Rubio de Francia's algorithm.

\begin{proposition}%[\cite{C-W}]
\label{RubiodeFrancia}
Given $p(\cdot) \in \mathcal{P}(\Xx)$, suppose $\mu$ is a weight such that $M$ is bounded on $L^{p(\cdot)}_{\mu}(\Xx)$. For a positive function $h \in L^1_{loc}(\Xx)$, with $Mh(x)<\infty$ almost everywhere, define
\begin{equation*}
\mathcal{R}h(x)=\sum_{k=0}^{\infty} \frac{M^k h(x)}{2^k \Vert M \Vert_{L^{p(\cdot)}_{\mu}(\Xx)}^{k}}.
\end{equation*}
Then, for fixed constants $\alpha>0$ and $\beta \in \mathbb{R}$ and other weight $\omega$, define the operator $H$ by
\begin{equation*}
Hh=\mathcal{R}(h^{\alpha}\omega^{\beta})^{1/\alpha} \omega^{-\beta/\alpha}.
\end{equation*}
\begin{enumerate}
\item[(a)] $h(x)\leq H(x)$,
\item[(b)] Let $v=\omega^{\beta/\alpha}\mu^{1/\alpha}$. Then $H$ is bounded on $L^{\alpha p(\cdot)}_{v}(\Xx)$, with $\Vert H \Vert_{L^{\alpha p(\cdot)}_{v}(\Xx)}\leq 2 \Vert h \Vert_{L^{\alpha p(\cdot)}_{v}(\Xx)}$.
\item[(c)] $(Hh)^{\alpha}\omega^{\beta} \in A_1$ with $[(Hh)^{\alpha}\omega^{\beta}]_{A_1}\leq 2 \Vert M \Vert_ {L^{p(\cdot)}_{\mu}(\Xx)}$.
\end{enumerate}
\end{proposition}

The second Proposition is a property of the weights $A_{p(\cdot),q(\cdot)}$.

\begin{proposition}%[\cite{C-W}]
\label{property-weight}
Given $p(\cdot),q(\cdot)$ in $\mathcal{P}(\Xx)$, $1<p(x)\leq q(x)<\infty$,
suppose there exists $\sigma>1$ such that $\frac{1}{p(x)}-\frac{1}{q(x)}=%
\frac{1}{\sigma^{\prime }}$. Then $\omega \in A_{p(\cdot),q(\cdot)}$ if and
only if $\omega^{\sigma} \in A_{q(\cdot)/\sigma}$.
\end{proposition}

The third Proposition is a more general result than the Theorem \ref{extrapolation}. We will show the proof on homogeneous type spaces.
\begin{proposition}%[\cite{C-W}]
Let $p_0, q_0, \sigma$ and exponents $p(\cdot), q(\cdot)$ be as in the
statement of Theorem \ref{extrapolation}. Fix $\beta_1 \in \mathbb{R}$ and
choose any $s$ such that 
\begin{equation} \label{h1}
q_0 - q_{-}\left( \frac{q_0}{\sigma} -1\right) <s<\min\{q_0,q_{-}\}.
\end{equation}
Let $r_0=q_0/s$, and define $\alpha_1=s$ and $\beta_2=s-\beta_1 (1-r_0)$.
Then if $\omega$ is a weight such that $M$ is bounded on $%
L^{(q(\cdot)/s)^{\prime }}(\omega^{-\beta_2})$, we have that 
\begin{equation}\label{r1}
\lVert f \rVert_{L^{q(\cdot)}_{\omega}(\Xx)} \leq C\lVert g \rVert_{L^{p(\cdot)}_{\omega}(\Xx)},
\end{equation}
for any $(f,g)\in \mathcal{F}$.
\end{proposition}

\begin{proof}
Fix a pair $(f,g)\in \mathcal{F}$, and we may assume without loss of
generality that $0<\left\Vert f\right\Vert _{L^{q(\cdot )}_{\omega}(\Xx)},~\left\Vert g\right\Vert _{L^{p(\cdot)}_{\omega}(\Xx)}<\infty $. Moreover, if 
$(f,g)$ satisfies (\ref{r1}) then so does $(\lambda f,\lambda g)$, for any $%
\lambda >0$, so without loss of generality we may assume that $\left\Vert
g\right\Vert _{L^{p(\cdot )}_{\omega}(\Xx)}=1$. Then by Remark \ref{modular-norma} it will
suffice to prove that $\left\Vert f\right\Vert _{L^{q(\cdot )}_{\omega}(\Xx)}\leq
C$.
Define
\begin{equation*}
h_{1}=\frac{f}{\left\Vert f \right\Vert _{L^{q(\cdot)}_{\omega}(\Xx)}}+g^{\frac{p(\cdot
)}{q(\cdot)}}\omega ^{\frac{p(\cdot )}{q(\cdot )}-1}\text{;}
\end{equation*}
we claim that $\left\Vert h_1 \right\Vert _{L^{q(\cdot)}_{\omega}(\Xx)}\leq C$. This
follows from Remark \ref{modular-norma}:
\begin{equation*}
\rho _{q(\cdot )}(h_{1}\omega )\leq 2^{q_{+}}\int_{\Xx}\left( 
\frac{f(x)\omega (x)}{\left\Vert f \right\Vert _{L^{q(\cdot)}_{\omega}(\Xx)}}\right)
^{q(x)}dx+2^{q_{+}}\int_{\Xx}\left[ g(x)\omega (x)\right]
^{p(x)}dx\leq 2^{q_{+}+1}
\end{equation*}
we again use Proposition \ref{RubiodeFrancia} to define two operators $H_{1}$ and $H_{2}$
as follow,
\begin{eqnarray*}
H_{1} &=&\mathcal{R}_{1}(h_{1}^{\alpha _{1}}\omega ^{\beta _{1}})^{1/\alpha
_{1}}\omega ^{-\beta _{1}/\alpha _{1}}\text{,} \\
H_{2} &=&\mathcal{R}_{2}(h_{2}^{\alpha _{2}}\omega ^{\beta _{2}})^{1/\alpha
_{2}}\omega ^{-\beta _{2}/\alpha _{2}}\text{.}
\end{eqnarray*}
Let \thinspace $r_{0}=q_{0}/s$, and fix $s$, $0<s<\min\{q_{0},q_{-}\}$. Then
there exists $h_{2}\in L^{\left( q(\cdot )/s\right)^{\prime}}(\Xx)$, $%
\left\Vert h_{2}\right\Vert _{L^{\left( q(\cdot )/s\right)^{\prime}}(\Xx)}=1$,
such that for any $\gamma >0$,%
\begin{eqnarray*}
\left\Vert f \right\Vert _{L^{q(\cdot)}_{\omega}(\Xx)}^{s} &\leq &C\int_{\Xx}f^{s}\omega ^{s}h_{2}dx \\
 &\leq &C\int_{\Xx}f^{s}H_{1}^{\gamma }H_{1}^{~-\gamma
}H_{2}\omega ^{s}dx \\
 &\leq &C\left( \int_{\Xx}f^{q_{0}}H_{1}^{~-\gamma
(q_{0}/s)}H_{2}\omega ^{s}dx\right) ^{1/r_{0}}\left( \int_{\Xx}H_{1}^{\gamma r_{0}^{\prime }}\omega ^{s}H_{2}dx\right)^{1/r_{0}^{\prime }} \\
&=&C(I_{1})^{1/r_{0}}(I_{2})^{1/r_{0}^{\prime }}
\end{eqnarray*}
We finding conditions to insure that $I_{2}$ is uniformly bounded. Since $%
h_{1}\in L^{q(\cdot)}_{\omega}(\Xx)$ and $h_{2}\in L^{\left( q(\cdot )/s\right)
^{\prime }}(\Xx)$ we require $H_{1}$ and $H_{2}$ to be bounded on these spaces.
We apply H\"{o}lder's inequality, Lemma \ref{HI}, with exponent $q(\cdot )/s$ to get%
\begin{equation*}
I_{2}\leq C\left\Vert H_{1}^{\gamma (q_{0}/s)^{\prime }}\omega
^{s}\right\Vert _{L^{q(\cdot )/s}(\Xx)}\left\Vert H_{2}\right\Vert _{L^{\left( q(\cdot
)/s\right) ^{\prime }}(\Xx)}\text{.}
\end{equation*}
If we let $\gamma =\frac{s}{(q_{0}/s)^{\prime }}$, then by Lemma \ref{dilationExpo}, 
\begin{eqnarray*}
\left\Vert H_{1}^{\gamma r_{0}^{\prime }}\omega ^{s}\right\Vert _{L^{q(\cdot
)/s}(\Xx)} &=&\left\Vert H_{1}\omega \right\Vert _{L^{q(\cdot)}(\Xx)}^{s}\leq
2^{s}\left\Vert h_{1}\omega \right\Vert _{L^{q(\cdot)}(\Xx)}^{s}\leq C\text{,} \\
\left\Vert H_{2}\right\Vert _{L^{\left( q(\cdot )/s\right)^{\prime }}(\Xx)} &\leq
&2\left\Vert h_{2}\right\Vert _{L^{\left( q(\cdot )/s\right) ^{\prime }}(\Xx)}=2\text{%
.}
\end{eqnarray*}
For $H_{1}$ and $H_{2}$ to be bounded on these spaces, by Proposition \ref{RubiodeFrancia}
we must have the maximal operator satisfies $M$ bounded on $L^{q(\cdot
)/\alpha _{1}}_{\omega ^{\alpha _{1}-\beta _{1}}}(\Xx)$ and $L^{\left( q(\cdot
)/s\right)^{\prime }/\alpha _{2}}_{\omega^{-\beta _{2}}}(\Xx)$.
For these to hold we must have that%
\begin{equation}\label{q-+}
q_{-}>\alpha _{1}\text{ \ and \ }(q_{+}/s)^{\prime }>\alpha _{2}\text{.}
\end{equation}
It remains to estimate $I_{1}$; with our value of $\gamma $ we now have that%
\begin{equation*}
I_{1}=\int_{\Xx}f^{q_{0}}H_{1}^{-q_{0}/r_{0}^{\prime
}}H_{2}\omega ^{s}dx\text{.}
\end{equation*}
In order to apply (\ref{conditionbase}) we need to show that $I_{1}$ is finite. However,
this follows from H\"{o}lder's inequality and the above estimates for $H_{1}$
and $H_{2}$.
\begin{equation*}
I_{1}\leq \left\Vert f\right\Vert _{L^{q(\cdot )}_{\omega}(\Xx)}^{q_{0}}\int_{\Xx}
H_{1}^{s}H_{2}\omega ^{s}dx\leq \left\Vert f\right\Vert _{L^{q(\cdot )}_{\omega}(\Xx)}^{q_{0}}\left\Vert H_{1}^{s}\omega ^{s}\right\Vert _{q(\cdot
)/s}\left\Vert H_{2}\right\Vert _{\left( q(\cdot )/s\right) ^{\prime
}}<\infty \text{.}
\end{equation*}
To apply our hypothesis (\ref{conditionbase}) we need the weight $\omega _{0}=\left(
H_{1}^{-\gamma (q_{0}-s)}H_{2}\omega ^{s}\right) ^{1/q_{0}}$ to be in $%
A_{p_{0},q_{0}}$, or equivalently by Proposition \ref{property-weight}, $\omega
^{q_{0}}=H_{1}^{-\gamma (q_{0}-s)}H_{2}\omega ^{s}\in A_{r_{1}}$, where $%
r_{1}=1+\frac{q_{0}}{p_{0}^{\prime }}=\frac{q_{0}}{\sigma }$.
To apply Lemma \ref{propA} and Proposition \ref{RubiodeFrancia} we write 
\begin{equation*}
\omega ^{q_{0}}=\left( H_{1}^{\frac{q_{0}-s}{r_{1}-1}}\omega ^{\beta
_{1}}\right) ^{1-r_{1}}H_{2}\omega ^{s-\beta _{1}(1-r_{1})}\text{.}
\end{equation*}
This gives the following constraints on $\alpha
_{j},\beta _{j}$:
\begin{equation*}
\alpha _{1}=\frac{q_{0}-s}{\frac{q_{0}}{\sigma }-1}\text{, }\beta _{1}\in 
\mathbb{R}\text{, }\alpha _{2}=1\text{, }\beta _{2}=s-\beta
_{1}(1-q_{0}/\sigma )
\end{equation*}
If we combine these with the constraints in (\ref{q-+}) we see that the second
one there always holds and the first one holds if
\begin{equation*}
s>q_{0}-q_{-}\left( \frac{q_{0}}{\sigma }-1\right) \text{.}
\end{equation*}
We can now apply (\ref{conditionbase}): by the definition of $h_{1}$ and by H\"{o}lder's
inequality with respect to the undetermined exponent $t(\cdot )$, we
get
\begin{eqnarray*}
I_{1}^{1/q_{0}} &\leq &C\left( \int_{\Xx}g^{p_{0}}\left[
H_{1}^{-q_{0}/r_{0}^{\prime }}\omega ^{s}H_{2}\right] ^{p_{0}/q_{0}}dx
\right) ^{1/p_{0}} \\
&\leq &C\left( \int_{\Xx}\left( h_{1}^{\frac{q(\cdot )}{p(\cdot )}
}\omega ^{\frac{q(\cdot )}{p(\cdot )}-1}\right)
^{p_{0}}H_{1}^{-p_{0}/r_{0}^{\prime }}H_{2}^{p_{0}/q_{0}}\omega
^{sp_{0}/q_{0}}dx\right) ^{1/p_{0}} \\
&\leq &C\left( \int_{\Xx}H_{1}^{p_{0}\left( \frac{q(\cdot )}{
p(\cdot )}-\frac{1}{r_{0}^{\prime }}\right) }H_{2}^{p_{0}/q_{0}}\omega
^{p_{0}\left( \frac{s}{q_{0}}+\frac{q(\cdot )}{p(\cdot )}-1\right)
}dx\right) ^{1/p_{0}} \\
&\leq &C\left\Vert H_{1}^{p_{0}\left( \frac{q(\cdot )}{p(\cdot )}-\frac{1}{%
r_{0}^{\prime }}\right) }\omega ^{p_{0}\left( \frac{q(\cdot )}{p(\cdot )}-%
\frac{1}{r_{0}^{\prime }}\right) }\right\Vert _{L^{t^{\prime }(\cdot
)}(\Xx)}^{1/p_{0}}\left\Vert H_{2}^{p_{0}/q_{0}}\right\Vert _{L^{t(\cdot
)}(\Xx)}^{1/p_{0}} \\
 &=&CJ_{1}^{1/p_{0}}J_{2}^{1/p_{0}}\text{.}
\end{eqnarray*}
If we let $t(\cdot )=\frac{q_{0}(q(\cdot )/s)^{\prime }}{p_{0}}$, then
by dilation $J_{2}$ is uniformly bounded. To show that $J_{1}$ is uniformly
bounded we first note that
\begin{equation*}
p_{0}\left( \frac{q(\cdot )}{p(\cdot )}-\frac{1}{r_{0}^{\prime }}\right)
t^{\prime }(\cdot )=q(\cdot )\text{.}
\end{equation*}
Given this, then
\begin{equation*}
\rho _{t^{\prime }(\cdot )}\left( H_{1}^{p_{0}\left( \frac{q(\cdot )}{%
p(\cdot )}-\frac{1}{r_{0}^{\prime }}\right) }\omega ^{p_{0}\left( \frac{%
q(\cdot )}{p(\cdot )}-\frac{1}{r_{0}^{\prime }}\right) }\right) =\int_{%
\Xx}H_{1}^{q(\cdot )}\omega ^{q(\cdot )}dx=\rho _{q(\cdot
)}\left( H_{1}\omega \right) \text{.}
\end{equation*}
If we apply Remark \ref{modular-norma} twice, since $\left\Vert H_{1}\right\Vert
_{L^{q(\cdot)}_{\omega}(\Xx)}\leq 2\left\Vert h_{1}\right\Vert _{L^{q(\cdot)}_{\omega}(\Xx)}$ is uniformly bounded, $\rho _{q(\cdot )}\left( H_{1}\omega
\right) $ is a well, and hence, $J_{1}$ is uniformly bounded. This completes
the proof.
\end{proof}

\begin{proof}[Proof of Theorem \ref{extrapolation}]
To prove Theorem \ref{extrapolation} we take $\beta_1=0$ and $s=\sigma$. We have that
\begin{equation*}
1-\frac{1}{\sigma} = \frac{1}{p_0}-\frac{1}{q_0} = \frac{1}{p_-}-\frac{1}{q_-},
\end{equation*}
so the second inequality in \ref{h1} holds. The first inequality is equivalent to 
\begin{equation*}
\sigma^2 -(q_0+q_-)\sigma + q_{-}q_0 > 0,
\end{equation*}
which follows from the second inequality. The requeriment on the weight $\omega$ reduces to $M$ being bounded on $L^{\frac{q(\cdot)}{\sigma}}_{\omega^{\sigma}}(\Xx)$ and $L^{(\frac{q(\cdot)}{\sigma})'}_{\omega^{-\sigma}}(\Xx)$ or, equivalently, $(q(\cdot)/\sigma, \omega^{\sigma})$ is an $M-pair$.
\end{proof}

We now state a result, given in \cite{SW}, for us to test the condition (\ref{conditionbase}).
\begin{lemma}[\cite{SW}]\label{lemmafi} 
Suppose $1<p\leq q< \infty$, $(\Xx,d)$ is a quasi-metric
space, $\mu$ is a doubling measure on $\Xx$, and $\omega(x)$ and $\nu(x)$ are
nonnegative $\mu$-measurable functions on $\Xx$. Let $\varphi(B)$ be given by 
\begin{equation*}
\varphi(B)=\sup \left\lbrace K(x,y): \ x,y \in B, d(x,y)\geq C(K)r(B)
\right\rbrace,
\end{equation*}
where $K(x,y)$ is the kernel of $I_{\alpha}$, $r(B)$ is the radius of $B$
and $C(K)=K^{-4}/9$, with $K$ the constant given in the equation \eqref{k}. If $p<q$, $\omega d\mu$ and $\nu^{1-p^{\prime }}d\mu$
are doubling measures, then the following weighted inequality 
\begin{equation*}
\left( \int_\Xx [I_{\alpha}f(x)]^q \omega(x)d\mu(x) \right)^{\frac{1}{q}} \leq
\left( \int_\Xx f(x)^p \nu(x)d\mu(x) \right)^{\frac{1}{p}}
\end{equation*}
holds if the condition 
\begin{equation*}  \label{conditionIalfa}
\varphi(B) \left( \int_B \omega d\mu \right)^{\frac{1}{q}} \left( \int_B
\nu^{1-p^{\prime }} d\mu \right)^{\frac{1}{p^{\prime }}} \leq C
\end{equation*}
holds for all balls $B\subset \Xx$.
\end{lemma}

\begin{corollary}
Let $p,q$ be as in Lemma \ref{lemmafi}. And let $\omega $ a weight in $A_{p,q}$. Then 
\begin{equation}
\left( \int_{\Xx}[I_{\alpha }f(x)]^{q}\omega (x)^{q}dx\right) ^{\frac{1}{q}%
}\leq \left( \int_{\Xx}f(x)^{p}\omega (x)^{p}dx\right) ^{\frac{1}{p}}.
\end{equation}
\end{corollary}

\begin{proof}
Let $B$ a ball on $\Xx$. We take $\omega=\omega^q$ and $\nu=\omega^p$ in Lemma
\ref{lemmafi}. Since $\omega \in A_{p,q}$ and $\varphi(B)\leq C \lvert B
\rvert^{\frac{\alpha}{Q}-1} $, we have 
\begin{equation*}
\varphi(B) \left( \int_B \omega(x)^q dx \right)^{\frac{1}{q}} \left( \int_B
[\omega(x)^p]^{1-p^{\prime }} dx \right)^{\frac{1}{p^{\prime }}}
\end{equation*}
\begin{equation*}
\leq C \lvert B \rvert^{\frac{\alpha}{Q}-1} \left( \int_B \omega(x)^q dx
\right)^{\frac{1}{q}} \left( \int_B \omega(x)^{-p^{\prime }} dx \right)^{%
\frac{1}{p^{\prime }}}
\end{equation*}
\begin{equation*}
\leq C \lvert B \rvert^{\frac{\alpha}{Q}-1} \lvert B \rvert^{\frac{1}{q}+%
\frac{1}{p^{\prime }}} \left( \frac{1}{\lvert B \rvert} \int_B \omega(x)^q
dx \right)^{\frac{1}{q}} \left( \frac{1}{\lvert B \rvert} \int_B
\omega(x)^{-p^{\prime }} dx \right)^{\frac{1}{p^{\prime }}}
\end{equation*}
\begin{equation*}
\leq C.
\end{equation*}
So by Lemma \ref{lemmafi} we have the desired result.
\end{proof}

And now applying the Extrapolation Theorem \ref{extrapolation} we have the following result.

\begin{theorem}\label{I_alpha_bounded-weight} 
Fix $\alpha$, $0<\alpha<Q$. Given $p(\cdot)
\in \mathcal{P}(\Xx)$ such that $1<p_{-}\leq p_{+}<\frac{Q}{\alpha}$, define $q(\cdot)$ by 
\begin{equation*}
\frac{1}{p(x)}-\frac{1}{q(x)}=\frac{\alpha}{Q}.
\end{equation*}
Let $\sigma=(Q/\alpha)^{\prime }$. If $\omega \in A_{p(\cdot),q(\cdot)}$ is
such that $(q(\cdot)/\sigma, \omega^{\sigma})$ is a $M$-pair, we have 
\begin{equation*}
\lVert I_{\alpha}f \rVert_{L^{q(\cdot)}_\omega(\Xx)} \leq C \lVert f
\rVert_{L^{p(\cdot)}_\omega(\Xx)}.
\end{equation*}
\end{theorem}

To prove the Theorem \ref{prin}, when $p^-=1$, we need the weighted weak estimates for the fractional integral operator in homogeneous spaces considering weights in the  Muckenhoupt class $A_{p,q}.$  This follows from the next result about the weak type estimates in spaces of homogeneous type.

\begin{theorem}[\cite{FGW},\cite{FLW}]\label{I_alpha_debil_constante}
Let $(\Xx,d,dx)$ be a Carnot-Carath\'eodory space. Let $1\leq p<q<\infty$, $K$ be a compact subset of $\Omega$, $B=B(x_0,r)$, $x_0\in  K$ and
\begin{equation*}
T_Bf(x)=\int_{B} \frac{d(x,y)^\alpha}{|B(x,d(x,y))|}|f(y)|dy,
\end{equation*}
where $f \in L^1_{\text{loc}}(\Xx)$ and $x\in B$.  There are a constants $r_0$ and $C$ depending only on $K$, $\Omega$ and the Carnot-Carath\'eodory metric, such that if $r<r_0$ and   
$u,v$ are weights, then
\begin{equation}
\left(t^q\int_{\{y\in B: Tf(y)>t\}} v(y)dy\right)^{1/q}\leq C  L \left(\int_{B}|f(y)|^pu(y)dy\right)^{1/p},\qquad t>0,
\end{equation}
with
\begin{eqnarray}
\label{L}
L &=&
\left\{ 
\begin{array}{c}
\sup \left( \int_{B(x,r)}vdy \right)^{1/q}\left(\int_{B}K_r(x,y)^{p'}u(y)^{-\frac{1}{p-1}}dy\right)^{1/p^{\prime }}\,\text{, if }\,p>1 \\ 
\sup \left( \int_{B(x,r)}vdy \right)^{1/q}\left( \text{ess sup}_{y\in B}K_r(x,y)u^{-1}(y)\right)\, \text{ , if }\,p=1\text{,}%
\end{array}%
\right. 
\end{eqnarray}
here
$$
K_r(x,y)=\min\left\{\frac{r^\alpha}{|B(x,r)|},\frac{d(x,y)^\alpha}{|B(x,(d(x,y)))|}\right\},
$$
and the supremum is taken over all $x$ and $r$ such that $B(x,r)\subset 5B$ and $x\in B$.
\end{theorem}
\begin{corollary}\label{I_alpha_bebil} 
Let be $\Omega$ a compact open set of $\Xx$. Fix $\alpha$, $0<\alpha<Q$. Suppose $1\leq p<q<\frac{Q}{\alpha}$ such that
\begin{equation*}
\frac{1}{p}-\frac{1}{q}=\frac{\alpha}{Q}.
\end{equation*}
Let $\omega\in A_{p,q}$, then exist a constat $c>0$ such that  
\begin{equation*}
\left(\int_{\{x\in\Omega:|I_{\alpha}f(x)|>t\}}\omega(x)^q\,dx\right)^{\frac1q}\leq c\left(\frac1{t^p}\int_{\Omega}|f(x)|^p\omega(x)^p\right)^{\frac1p},
\end{equation*}
for all $f\in L^p_\omega(\Omega)$.
\end{corollary}
\begin{proof}
The proof follows from the Remark 4.3 in \cite{FGW}. In the case of $u=\omega^p$, $v=\omega^q$  and $\omega \in A_{p,q}$ with $\frac{1}{p}-\frac{1}{q}=\frac{\alpha}{Q}$ we have $L\leq C(\omega)$. Let's see this for the case of $p=1$, as $\omega \in A_{1,q}$ so $\omega^q dx$ is a doubling measure then
\begin{align*}
\text{ess sup}_{y\in B}K_r(x,y)\omega^{-1}(y)&\leq \text{ess sup}_{y\in B(x,r)}K_r(x,y)\omega^{-1}(y) \\
&\qquad+\sum_{j>0}\text{ess sup}_{y\in B\cap B(x,2^{j+1}r)\backslash B(x,2^{j}r)}K_r(x,y)\omega^{-1}(y)\\
&\leq C\sum_{j>0} \frac{(2^{j+1}r)^\alpha}{|B(x,2^{j+1}r)|}\text{ess sup}_{y\in B(x,2^{j+1}r)}\omega^{-1}(y)\\
%&\leq C\sum_{j>0} |B(x,2^{j+1}r)|^{\frac1Q-1}ess~\sup_{y\in B(x,2^{j+1}r)}\omega^{-1}(y)\\
&\leq C(\omega)\sum_{j>0}\left(\int_{B(x,2^{j+1}r)}\omega(y)^{q}dy\right)^{-1/q} ,  \qquad \omega \in A_{1,q}\\
&\leq C(\omega)\sum_{j>0}k^{(j+1)/q}\left(\int_{B(x,r)}\omega(y)^{q}dy\right)^{-1/q}  , \qquad \text{ for some }\, 0<k<1\\
&=C(\omega) \left(\int_{B(x,r)}\omega(y)^{q}dy\right)^{-1/q} ,
\end{align*}
then
\begin{align*}
L&=\sup \left( \int_{B(x,r)}\omega(y)^qdy \right)^{1/q}\left( \text{ess sup}_{y\in B}K_r(x,y)u^{-1}(y)\right)\\
 &\leq C(\omega) \sup \left( \int_{B(x,r)}\omega(y)^qdy \right)^{1/q}\left(\int_{B(x,r)}\omega(y)^{q}dy\right)^{-1/q}=C(\omega).
\end{align*}
Finally, when $\overline{\Omega}$ is a compact subset. There exists a ball $B_0$ such that $\Omega\subset B_0$. Then we apply Theorem \ref{I_alpha_debil_constante} to get the weak type estimates of $I_\alpha$ in $\Omega$.
\end{proof}

We recall a representation formula given by Lu and Wheeden in \cite{LW}. 

The measure $\mu$ satisfies a reverse doubling condition of order $1$, if there is a constant $C>0$ such that if $B_1$ and $B_2$ are balls with centres in $\Omega$ and with $B_1\subset B_2$, then
$$
\mu(B_2)\geq C\frac{\rho(B_2)}{\rho(B_1)}\mu(B_1),
$$
where $\rho (B)$ denotes the radius of $B$.

The 'segment' property holds for a ball $B\subset \Omega$ with center $x_B$, if for each $x\in B$ there  is  a  continuous  curve $\gamma:[0,1]\to B$ such that $\gamma(0)=x_B$, $\gamma(1)=x$ and $d(x_B,z)=d(x_B,y)+d(y,z)$ for all $y,\,z\in\gamma$ with $y=\gamma(s)$, $z=\gamma(t)$, $0\leq s\leq t\leq1.$

\begin{theorem}[\protect\cite{LW}]
\label{RepresentationFormulaFLW} Suppose that $\mu $ and $\nu $ are doubling
measures on a Carnot-Carath\'{e}odory space $(\mathcal{X},d)$, $\mu$ satisfies a reverse doubling condition of order $1$ and let $\Omega \subset \mathcal{X}$ be a weak Boman chain domain, such that the 'segment' property holds for all ball $B\subset \Omega$. Assume that there
exists $a_{1}\geq 1$ such that for all balls $B$ with $a_{1}B\subset \Omega $%
, 
\begin{equation*}
\dfrac{1}{\nu (B)}\int_{B}\lvert f-f_{B,\nu }\rvert d\nu \leq C\dfrac{\rho
(B)}{\mu (B)}\int_{a_{1}B}\lvert Xf\rvert d\mu ,
\end{equation*}%
where  $f_{B,\nu }=\frac{1}{\nu (B)}%
\int_{B}f(y)d\nu (y)$. Then for $\nu -a.e.x\in \Omega $, 
\begin{equation*}
\lvert f(x)-f_{B_{0},\nu }\rvert \leq C\int_{\Omega }\lvert Xf(y)\rvert 
\dfrac{d(x,y)}{\mu (B(x,d(x,y)))}d\mu (y),
\end{equation*}%
where $B_{0}$ is the central ball in $\Omega $, and $C$ is independent of $f$
and $x\in \Omega $.
\end{theorem}

The proof of the Theorem \ref{prin} follows from the constant exponent case, the last lemma and the Extrapolation Theorem.

\begin{lemma}
\label{Poincaré_p=1} Given a compact weak Boman chain domain $\Omega \subset X$ and 
$p$, $1\leq p<Q$, $\omega \in A_{p,p\ast }$, there is a constant $C=C(\Omega
,p,\omega )$ such that for all $f\in Lip(\Omega )$, 
\begin{equation*}
\left( \int_{\Omega }\lvert f(x)-f_{\Omega }\rvert ^{p\ast }\omega(x)^{p\ast
}dx\right) ^{1/p\ast }\leq C\left( \int_{\Omega }\lvert Xf(x)\rvert
^{p}\omega (x)^{p}dx\right) ^{1/p},
\end{equation*}%
where $p\ast =\dfrac{Qp}{Q-p}$.
\end{lemma}

\begin{proof}
Let $\Omega _{j}=\left\{ x\in \Omega :2^{j}<\lvert f(x)-f_{\Omega }\rvert
\leq 2^{j+1}\right\} $. We define%
\begin{equation*}
f_{j}(x)=\left\{ 
\begin{array}{ll}%
2^{j}  &\qquad \text{si }\left\vert f(x)-f_{\Omega }\right\vert \leq 2^{j} \\ 
\left\vert f(x)-f_{\Omega }\right\vert &\qquad \text{si }x\in \Omega _{j} \\ 
2^{j+1} &\qquad \text{si }\left\vert f(x)-f_{\Omega }\right\vert >2^{j+1}%
\end{array}%
\right.
\end{equation*}

It is easy to see that the function $f_{j}$ is weakly differetiable and $%
\left\vert Xf_{j}(x)\right\vert =\left\vert Xf(x)\right\vert \chi _{\Omega
_{j}}$, almost everywhere.

We have%
\begin{equation*}
2^{j}\leq f_{j}(x)\leq 2^{j}+\left\vert f(x)-f_{\Omega }\right\vert \text{.}
\end{equation*}

If $x\in \Omega _{j+1}$, by equation \eqref{poincare} we can apply the Theorem \ref{RepresentationFormulaFLW} considering the Lebesgue measure,
\begin{eqnarray*}
2^{j+1} &=&f_{j}(x)=\left\vert f_{j}(x)-f_{j,B}\right\vert +f_{j,B} \\
&\leq &C(\Omega) I_{1}(Xf_{j}(y))+2^{j}+\frac{1}{\left\vert B\right\vert }%
\int_{B}\left\vert f-f_{\Omega }\right\vert dz\text{.}
\end{eqnarray*}

And by equation \eqref{poincare} 
\begin{eqnarray*}
\frac{1}{\left\vert B\right\vert }\int_{B}\left\vert f-f_{\Omega
}\right\vert dz &\leq &\frac{C(\Omega )}{\left\vert B\right\vert }%
\int_{\Omega }\left\vert Xf\right\vert dz \\
%&\leq &C(B,\Omega)\int_{\Omega }\left\vert Xf(x)\right\vert \frac{\rho
%(x,y)}{\left\vert B(x,\rho (x,y))\right\vert }dy \\
%&\leq &C(\Omega)I_{1}(Xf)(x)\text{.}
\end{eqnarray*}
%\begin{eqnarray*}
%\frac{1}{\left\vert B\right\vert }\int_{B}\left\vert f-f_{\Omega
%}\right\vert dz &\leq &\frac{l(\Omega )}{\left\vert B\right\vert }%
%\int_{\Omega }\left\vert Xf\right\vert dz \\
%&\leq &C(B,\Omega)\int_{\Omega }\left\vert Xf(x)\right\vert \frac{\rho
%(x,y)}{\left\vert B(x,\rho (x,y))\right\vert }dy \\
%&\leq &C(\Omega)I_{1}(Xf)(x)\text{.}
%\end{eqnarray*}
So,
\begin{equation*}
2^{j}\leq C(\Omega)I_{1}(Xf_{j})(x)+\frac{C(\Omega )}{\left\vert B\right\vert }\int_{\Omega }\left\vert Xf\right\vert dz 
\end{equation*}
%\begin{equation}\label{2jota}
%2^{j+1}=f_{j}(x)\leq C(\Omega)I_{1}(Xf_{j})(x)+2^{j}\text{.}
%\end{equation}
We choose $M$ such that
$$
2^{M-1}<\frac{C(\Omega )}{\left\vert B\right\vert }\int_{\Omega }\left\vert Xf\right\vert dz\leq 2^{M}. 
$$
If $j>M$ we get
\begin{equation}\label{2jota}
2^{j}\leq C(\Omega)I_{1}(Xf_{j})(x)+2^{j-1}. 
\end{equation}
Finally, we have%
\begin{eqnarray*}
\int_{\Omega }\left\vert f-f_{\Omega }\right\vert ^{p^{\ast }}\omega
^{p^{\ast }}dx &=& \int_{\{x\in \Omega: \left\vert f-f_{\Omega }\right\vert\leq 2^{M+1} \}}\left\vert f-f_{\Omega }\right\vert ^{p^{\ast }}\omega
^{p^{\ast }}dx\\
&&\qquad+\int_{\{x\in \Omega: \left\vert f-f_{\Omega }\right\vert> 2^{M+1}\} }\left\vert f-f_{\Omega }\right\vert ^{p^{\ast }}\omega
^{p^{\ast }}dx\\
&=&\int_{\{x\in \Omega: \left\vert f-f_{\Omega }\right\vert\leq 2^{M+1} \}}2^{(M+1)p^{\ast}}\omega
^{p^{\ast }}dx\\
&&\qquad+\sum_{j>M}\int_{\Omega _{j}}\left\vert f(x)-f_{\Omega
}\right\vert ^{p^{\ast }}\omega(x) ^{p^{\ast }}dx \\
&=&2^{(M+1)p^{\ast }}\omega
^{p^{\ast }}(\Omega)+\sum_{j>M}\int_{\Omega _{j}}(2^{j+1})^{p^{\ast }}\omega(x) ^{p^{\ast }}dx
\end{eqnarray*}
%\begin{eqnarray*}
%\int_{\Omega }\left\vert f-f_{\Omega }\right\vert ^{p^{\ast }}\omega
%^{p^{\ast }}dx &=&\sum_{j}\int_{\Omega _{j}}\left\vert f(x)-f_{\Omega
%}\right\vert ^{p^{\ast }}\omega(x) ^{p^{\ast }}dx \\
%&=&\sum_{j}\int_{\Omega _{j}}(2^{j+1})^{p^{\ast }}\omega(x) ^{p^{\ast }}dx
%\end{eqnarray*}
On the one hand, by  \eqref{2jota} in $\Omega _{j+1}$, if $j>M$ we have $C2^{j-1}\leq I_{1}(Xf_{j})(x)$. And by Corollary \ref{I_alpha_bebil},
\begin{eqnarray*}
\sum_{j>M}\int_{\Omega _{j}}(2^{j+1})^{p^{\ast }}\omega(x)
^{p^{\ast }}dx &=&4^{p^{\ast }}C(\Omega)^{-p^{\ast }}\sum_{j}\int_{\Omega _{j}}\left(
C^{-1}2^{j-1}\right) ^{p^{\ast }}\omega(x) ^{p^{\ast }}dx \\
&\leq &C(\Omega,p)\sum_{j}\int_{\{x\in \Omega :I_{1}(\left\vert Xf_{j-1}\right\vert
)(x)>C2^{j-1}\}}\left( C2^{j-1}\right) ^{p^{\ast }}\omega(x) ^{p^{\ast }}dx
\\
&\leq &C(\Omega,p)\sum_{j}\left( \int_{\Omega }\left\vert Xf_{j-1}(x)\right\vert
^{p}\omega (x)^{p}dx\right) ^{\frac{p^{\ast }}{p}} \\
%&=&C(\Omega,p)\sum_{j}\left( \int_{\Omega _{j-1}}\left\vert Xf(x)\right\vert ^{p}\omega
%(x)^{p}dx\right) ^{^{\frac{p^{\ast }}{p}}} \\
&\leq &C(\Omega,p)\left( \int_{\Omega }\left\vert Xf(x)\right\vert ^{p}\omega
(x)^{p}dx\right) ^{^{^{\frac{p^{\ast }}{p}}}}\text{.}
\end{eqnarray*}

On the other hand
\begin{align*}
2^{(M+1)}\left(\omega^{p^{\ast}}(\Omega)\right)^{1/p^*}&\leq 4C\left(\omega^{p^{\ast}}(\Omega)\right)^{1/p^*}\frac{C(\Omega )}{\left\vert B\right\vert }\int_{\Omega }\left\vert Xf\right\vert dz\\
&\leq4C\left(\omega^{p^{\ast}}(\Omega)\right)^{1/p^*}\frac{C(\Omega )}{\left\vert B\right\vert }\left(\omega^{-p'}(\Omega)\right)^{1/p'}\left(\int_{\Omega }\left\vert Xf\right\vert^p\omega(x)^{p} dz\right)^{\frac{1}{p}}
\end{align*}
Therefore, we get 
\begin{equation*}
\left( \int_{\Omega }\lvert f(x)-f_{\Omega }\rvert ^{p\ast }\omega(x)^{p\ast
}dx\right) ^{1/p\ast }\leq C\left( \int_{\Omega }\lvert Xf(x)\rvert
^{p}\omega (x)^{p}dx\right) ^{1/p},
\end{equation*}
where the constant $C$ depends of $\omega$, $\Omega$ and $p$.
\end{proof}

We can now prove the Theorem \ref{prin}

\begin{proof}[Proof of Theorem \ref{prin}]
Case 1: If $p_{-}>1$, let $B\subset \Omega $ be a ball, 
\begin{equation*}
\lVert f-f_{\Omega }\rVert _{L^{p\ast (\cdot )}_\omega(\Omega )}\leq \lVert
f-f_{B}\rVert_{L^{p\ast (\cdot )}_\omega(\Omega )}+\lVert f_{B}-f_{\Omega }\rVert
_{L^{p\ast (\cdot )}_\omega(\Omega )}.
\end{equation*}%
By Theorem \ref{HI} (H\"older's inequality)  
\begin{align*}
\lVert f_{B}-f_{\Omega }\rVert_{L^{p\ast (\cdot )}_\omega(\Omega )}&=\lvert
f_{B}-f_{\Omega }\rvert \lVert \omega \rVert_{L^{p\ast (\cdot )}(\Omega)}\\
&\leq \left( \dfrac{1}{\lvert \Omega \rvert }\int_{\Omega }\lvert
f_{B}-f\rvert \chi _{\Omega }\omega \omega ^{-1}dx\right) \lVert \omega \rVert_{L^{p\ast (\cdot )}(\Omega)}\\
&\leq \dfrac{1}{\lvert \Omega \rvert }\lVert f-f_{B}\rVert_{L^{p\ast
(\cdot )}_\omega(\Omega)}\lVert \omega^{-1} \rVert_{L^{(p\ast (\cdot ))'}(\Omega)}\lVert \omega \rVert_{L^{p\ast (\cdot )}(\Omega)}.
\end{align*}
Let $W(x)=\omega (x)^{p^{\ast }(x)}$, $W(\Omega)=\int_{\Omega}W(x)dx$ and if $\lambda =W(\Omega )+1$ then 
\begin{equation*}
\int_{\Omega }\lambda ^{p^{\ast }(x)}\omega(x) ^{p^{\ast }(x)}dx\leq \lambda
^{-p_{-}^{\ast }}W(\Omega )\leq \lambda ^{-1}(W(\Omega )+1)=1.
\end{equation*}%
Therefore $\lVert \omega \rVert _{L^{p\ast (\cdot )}(\Omega)}\leq
W(\Omega )+1$. Analogously for $\lVert \omega ^{-1}\rVert
_{L^{(p\ast (\cdot ))^{\prime }}(\Omega)}$. So,

\begin{equation*}
\dfrac{1}{\lvert \Omega \rvert} \lVert f-f_{B} \rVert_{L^{p*(\cdot)}_\omega(\Omega)}\lVert \omega^{-1} \rVert_{L^{(p\ast (\cdot ))'}(\Omega)}\lVert \omega \rVert_{L^{p\ast (\cdot )}(\Omega)} \leq C(\Omega, \omega) \lVert
f-f_{B} \rVert_{L^{p*(\cdot)}_\omega(\Omega)}.
\end{equation*}
Then, 
\begin{equation}  \label{ballnorm}
\lVert f-f_{\Omega}\rVert_{L^{p*(\cdot)}_\omega(\Omega)} \leq C(\Omega,\omega) \lVert
f-f_{B}\rVert_{L^{p*(\cdot)}_\omega(\Omega)}.
\end{equation}

Now, we consider $B_{0}$ the central ball in $\Omega$, by Theorem \ref{RepresentationFormulaFLW} with $\mu$ and $\nu$ the Lebesgue measures

\begin{equation}  \label{x}
\lvert f-f_{B_0} \rvert \leq C \int_{\Omega} \lvert Xf(y) \rvert \dfrac{%
\rho(x,y)}{\lvert B(x, \rho (x,y)) \rvert} dy \leq C I_1 (\lvert Xf \rvert)
(x).
\end{equation}

Finally by Theorem \ref{I_alpha_bounded-weight}, equation \eqref{ballnorm} and 
\eqref{x} we have

\begin{equation*}
\lVert f-f_{\Omega }\rVert_{L^{p*(\cdot)}_\omega(\Omega)}\leq C\lVert
I_{1}(Xf)\rVert_{L^{p*(\cdot)}_\omega(\Omega)}\leq C\lVert Xf\rVert
_{L^{p(\cdot)}_\omega(\Omega)}.
\end{equation*}

We observe that the same result is obtained by extrapolation techniquess
from the inequality of Lemma \ref{Poincaré_p=1}.

Case 2: If $p_{-}=1$ by the assumption $M$ bounded on $L^{\left( p\ast
(\cdot )/Q^{\prime }\right) ^{\prime }}(\omega ^{-Q^{\prime }})$. By Lemma %
\ref{Poincaré_p=1}, we have for $\omega _{0}\in A_{1,Q^{\prime }}$ 
\begin{equation*}
\left( \int_{\Omega }\lvert f(x)-f_{\Omega }\rvert ^{p\ast }\omega
_{0}(x)^{p\ast }dx\right) ^{1/p\ast }\leq C\left( \int_{\Omega }\lvert
Xf(x)\rvert ^{p}\omega _{0}(x)^{p}dx\right) ^{1/p},
\end{equation*}%
then we can applied Theorem \ref{extrapolation} and we obtain the desired
result.
\end{proof}

In \cite{LW-2} and \cite{LW-3} for Carnot group the authors establish
high-order Sobolev embedding theorems, proving high order representation
formulas for smooth functions. In fact, they use the existence and
properties of polynomials given in \cite{L1} and \cite{L2}, which only
require the existence of the distributional derivative. 
%For this, the above high order representation goes through for weak derivative functions as well, see \cite{LLT}. 

\begin{theorem}[\protect\cite{LW-3}]
\label{P} Let $\Omega $ be a weak Boman domain in a Carnot group $\mathbb{G}$
with central ball $B_{0}$, and let $f\in W_{loc}^{m,1}(\Omega )$. Let $Q$ be
the homogeneous dimension of $\mathbb{G}$. Then for any integers $0\leq
j<i\leq m$ with $i-j\leq Q$, there is a polynomial $P_{m}(B_{0},f)$ of
homogeneous order less than $m$ such that for $a.e.x\in \Omega $, 
\begin{equation*}
|X^{j}(f-P_{m}(B_{0},f))(x)|\leq C\int_{\Omega }|X^{i}(f)(y)|\frac{%
d(x,y)^{i-j}}{|B(x,d(x,y))|},
\end{equation*}%
where $C$ is independent of $f$.
\end{theorem}

\begin{theorem}[\protect\cite{LW-3}]
\label{reprG} Let $\Omega $ be a domain in a Carnot group $\mathbb{G}$, and
let $Q$ be the homogeneous dimension of $\mathbb{G}$. Suppose that $m$ and $j
$ are integers with $0\leq j<m$ and $m-j\leq Q$. If $f\in W_{0}^{m,1}(\Omega
)$, then for any $m<Q$ and for $a.e.x\in \mathbb{G}$, 
\begin{equation*}
\lvert X^{j}f(x)\rvert \leq C\int_{\Omega }\lvert X^{m}f(y)\rvert \dfrac{%
d(x,y)^{m-j}}{\lvert B(x,d(x,y))\rvert }dy.
\end{equation*}
\end{theorem}

The proof of the Theorem \ref{Poincare} and Theorem \ref{prin2} follows from the Theorem \ref{P} and Theorem \ref{reprG} respectively, applying the Theorem \ref{I_alpha_bounded-weight} with $\alpha=i-j$.

Finally we prove the Theorem \ref{Poincare2}

\begin{proof}[Proof of Theorem \ref{Poincare2}]
As $W^{1,p(\cdot)}_{\omega,0}(\Omega)=\overline{C_0^\infty(\Omega)}$ it is enough to consider $f\in C_0^\infty(\Omega)$. Since $\overline{\Omega} $ is compact, there exist $x_{1},...,x_{k}$ such that 
\begin{equation*}
\Omega =\cup _{i=1}^{k}B(x_{i},\delta ).
\end{equation*}%
We write $B_{i}=B(x_{i},\delta )$ and denote by $\chi _{i}$ the
characteristic function of $B_{i}$. Let $\tilde{f}$ the extension of $\tilde{%
f}$, i.e. $\tilde{f}(x)=0$ for $x\in \mathbb{G}-\Omega $. By Lemma \ref{propA} we get $\omega \in
A_{p^{-},(p^{-})^{\ast }}\Rightarrow \omega \in A_{p_{B_{i}}^{-},p_{B_{i}}^{\ast }}$,(remember that $p_{B_{i}}^{-}=\text{ess inf}_{x\in B_{i}}p(x)$). 
The weighted Poincar\'{e} inequality in the ball, Theorem \ref{prin},  imply that 
\begin{eqnarray*}
\lVert f\rVert _{L_{\omega }^{p(\cdot )}(\Omega )} &=&\lVert \tilde{f}\rVert
_{L_{\omega }^{p(\cdot )}(\mathbb{G})}\leq \sum_{i}\left\Vert \tilde{f}\chi
_{i}\right\Vert _{L_{\omega }^{p(\cdot )}(\mathbb{G})} \\
&\leq &C(\Omega)\sum_{i}\left\Vert \tilde{f}\right\Vert _{L_{\omega
}^{p_{B_{i}}^{\ast }}(B_{i})} \\
&\leq &C(\Omega)\sum_{i}\left( \left\Vert \tilde{f}-\tilde{f}%
_{B_{i}}\right\Vert_{L_{\omega }^{p_{B_{i}}^{\ast }}(B_{i})}+\left\vert 
\tilde{f}_{B_{i}}\right\vert \left\Vert \chi _{i}\right\Vert
_{L_{\omega }^{p_{B_{i}}^{\ast }}(B_{i})}\right) \\
&\leq &C(\Omega)\sum_{i}\left( \left\Vert X\tilde{f}~\right\Vert
_{L_{\omega }^{p_{B_{i}}^{-}}(B_{i})}+\left\vert \tilde{f}%
_{B_{i}}\right\vert \left\Vert \omega \right\Vert_{L^{p^{\ast }(\cdot
)}(\Omega )}\right) .
\end{eqnarray*}

For every $i=1,...,k$ the classical Poincar\'{e} inequality on $\mathbb{G}$, equation \eqref{poincare'}, implies that%
\begin{eqnarray*}
\left\vert \tilde{f}_{B_{i}}\right\vert &\leq &C_{\delta }\int_{\Omega
}\left\vert f\right\vert dx\leq C\int_{\Omega }\left\vert Xf\right\vert dx \\
&\leq &C\left\Vert Xf~\right\Vert _{L_{\omega }^{p(\cdot )}(\Omega
)}\left\Vert \omega ^{-1}\right\Vert _{L^{p^{\prime }(\cdot )}(\Omega ).}
\end{eqnarray*}

So, 
\begin{equation*}
\lVert f\rVert _{L_{\omega }^{p(\cdot )}(\Omega )}\leq C\left\Vert
Xf~\right\Vert _{L_{\omega }^{p(\cdot )}(\Omega )}
\end{equation*}

where $C$ depend of $\Omega $, $p$ and $\omega $.
\end{proof}

\section{Proof of the applications.}\label{sec4}

In this section we will prove our results with respect to Dirichlet problems, equation \eqref{1}, for the degenerate $p(\cdot)$-Laplacian given in \eqref{dege}.

\begin{proof}[Proof Theorem \ref{carac}]
First we prove that if $u\in W^{1,p(\cdot)}_{\omega,0}(\Omega)$ minimizer the energy functional $\Ff$ then
 \begin{align}\label{menos}
	&\int_{\Omega }\left\langle A(x)Xu(x),Xu(x)\right\rangle ^{\frac{p(x)-2}{2}}	\left\langle A(x)Xu(x),Xv(x)\right\rangle \,dx\\ \nonumber
	&\qquad\qquad \qquad+\int_{\Omega}|u(x)|^{p(x)-2}u(x)v(x)\omega(x)^{p(x)}\,dx  -\int_{\Omega }f(x)v(x)\,dx\geq 0, \nonumber
	\end{align}
	for every $v\in W_{\omega,0}^{1,p(\cdot )}(\Omega )$. We fix $v\in W_{\omega,0}^{1,p(\cdot )}(\Omega )$, then for every $t \in\RR$ we get
	\begin{equation*}
\mathcal{F}(u+tv)-\mathcal{F}(u)\geq 0.
\end{equation*}
and  if we consider $ 0<t<1$
	\begin{equation}\label{sol}
\int_{\Omega}\frac{\mathcal{F}(u+tv)(x)-\mathcal{F}(u)(x)}{t}\,dx\geq 0.
\end{equation}
Since 
 \begin{align}\label{deri}
&\lim_{t\to 0} \left(\frac{\left\langle A(x)X (u+tv)(x),X(u+tv)(x)\right\rangle^{\frac{p(x)}{2}}-\left\langle A(x)X u(x),Xu(x)\right\rangle^{\frac{p(x)}{2}}}{p(x)\,t}\right.\\ \nonumber
&\qquad\quad\left.+\frac{|(u+tv)(x)|^{p(x)} -|u(x)|^{p(x)} }{p(x)\,t}+\frac{f(x)u(x)-f(x)(u+tv)(x)}{t}\right)\\ \nonumber
&\qquad=\left\langle A(x)Xu(x),Xu(x)\right\rangle ^{\frac{p(x)-2}{2}}\left\langle A(x)Xu(x),Xv(x)\right\rangle\\ \nonumber
	&\quad\qquad+|u(x)|^{p(x)-2}u(x)v(x)\omega(x)^{p(x)}-f(x)v(x), \nonumber
\end{align}
for almost every $x\in\Omega$, the result follows from the Lebesgue dominated convergence theorem provided that, if $A(x)$ satisfy the hypothesis \eqref{A}, we find a $L^1$ majorant independent of $t$ for integrand in \eqref{deri}.
By the mean value theorem there exists $\tilde t\in(0,t)$ such that
\begin{align*}
&\frac{\left\langle A(x)X (u+tv)(x),X(u+tv)(x)\right\rangle^{\frac{p(x)}{2}}-\left\langle A(x)X u(x),Xu(x)\right\rangle^{\frac{p(x)}{2}}}{p(x)\,t}\\ 
&\qquad\quad+\frac{|(u+tv)(x)\omega(x)|^{p(x)} -|u(x)\omega(x)|^{p(x)} }{p(x)\,t}+\frac{f(x)u(x)-f(x)(u+tv)(x)}{t}\\ 
&\qquad=\left\langle A(x)Xu(x),Xu(x)\right\rangle ^{\frac{p(x)-2}{2}}\left\langle A(x)Xu(x),\tilde{t} Xv(x)\right\rangle\\ 
	&\quad\qquad+|u(x)|^{p(x)-2}u(x)\omega(x)^{p(x)}\tilde{t}v(x)-f(x)\tilde{t}v(x), 
\end{align*}
and thus
\begin{align*}
&\left|\frac{\left\langle A(x)X (u+tv)(x),X(u+tv)(x)\right\rangle^{\frac{p(x)}{2}}-\left\langle A(x)X u(x),Xu(x)\right\rangle^{\frac{p(x)}{2}}}{p(x)\,t}\right.\\ 
&\qquad\quad\left.+\frac{|(u+tv)(x)\omega(x)|^{p(x)} -|u(x)\omega(x)|^{p(x)} }{p(x)\,t}+\frac{f(x)u(x)-f(x)(u+tv)(x)}{t}\right|\\ 
&\qquad\leq\tilde{t}\left(\left\langle A(x)Xu(x),Xu(x)\right\rangle ^{\frac{p(x)-2}{2}}|\left\langle A(x)Xu(x), Xv(x)\right\rangle|\right.\\ 
	&\quad\qquad\left.+|u(x)|^{p(x)-2}\omega(x)^{p(x)}|u(x)v(x)|+|f(x)v(x)|\right)\\
&\qquad\leq  |A(x)Xu(x)|^{\frac{p(x)}{2}}|Xu(x)|^{\frac{p(x)-2}{2}}|Xv(x)|+|u(x)|^{p(x)-1}\omega(x)^{p(x)}|v(x)|+|f(x)v(x)|\\
&\qquad\leq  \eta_2^{p(x)/2}\omega(x)^{p(x)}|Xu(x)|^{\frac{p(x)}{2}}|Xu(x)|^{\frac{p(x)-2}{2}}|Xv(x)|+|u(x)|^{p(x)-1}\omega(x)^{p(x)}|v(x)|+|f(x)v(x)|\\
	&\qquad\leq C(A,p)\omega(x)^{p(x)}|Xu(x)|^{p(x)-1}|Xv(x)|+\omega(x)^{p(x)}|u(x)|^{p(x)-1}|v(x)|+|f(x)||v(x)|=g(x),
\end{align*}
where $C(A,p)=\sup\{\eta_2^{p(x)/2}:x\in\Omega\}$.

Observe that $|Xu|^{p(\cdot)-1}\omega^{p(\cdot)}\in L^{p'(\cdot)}_{\omega^{-1}}(\Omega)$, if we take $\lambda=(\|Xu\|_{L^{p(\cdot)}_\omega(\Omega)}+1)^{p^--1}$,
\begin{align*}
\rho_{p'(\cdot),\omega^{-1}}(|Xu|^{p(\cdot)-1}\omega^{p(\cdot)}/\lambda)&=\int_{\Omega} \frac{|Xu(x)|^{(p(x)-1)p'(x)}\omega(x)^{p(x)p'(x)}}{\lambda^{p'(x)}}\omega(x)^{-p'(x)}\,dx\\
&\leq \int_{\Omega} \frac{|Xu(x)|^{(p(x)-1)p'(x)}}{(\|Xu\|_{L^{p(\cdot)}_\omega(\Omega)}+1)^{(p(x)-1)p'(x)}}\omega(x)^{p(x)p'(x)-p'(x)}\,dx\\
&=\int_{\Omega} \frac{|Xu(x)|^{p(x)}}{(\|Xu\|_{L^{p(\cdot)}_\omega(\Omega)}+1)^{p(x)}}\omega(x)^{p(x)}\,dx<1,
\end{align*}
then $|Xu|^{p(\cdot)-1}\omega^{p(\cdot)}\in L^{p'(\cdot)}_{\omega^{-1}}(\Omega)$ and  by the Lemma \ref{HI} (H\"older's inequalities)
\begin{align*}
\int_{\Omega}\omega(x)^{p(x)}|Xu(x)|^{p(x)-1}|Xv(x)|\,dx&=\int_{\Omega}\omega(x)^{p(x)}|Xu(x)|^{p(x)-1}|Xv(x)|\omega(x)^{-1}\omega(x)\,dx\\
&\leq c\||Xu|^{p(\cdot)-1}\omega(x)^{p(\cdot)}\|_{L^{p'(\cdot)}_{\omega^{-1}}(\Omega)}\|Xv(x)\|_{L^{p(\cdot)}_{\omega}(\Omega)},
\end{align*}
This implies that $\omega(x)^{p(x)}|Xu(x)|^{p(x)-1}|Xv(x)|\in L^1(\Omega)$. The similar way we see that  $\omega(x)^{p(x)}|u(x)|^{p(x)-1}|v(x)|\in L^1(\Omega)$ and $fv \in L^1(\Omega)$.
Therefore $g\in L^1(\Omega)$ is the desired majorant. 

Now, we prove that if $u \in W^{1,p(\cdot)}_{\omega,0}(\Omega)$ satisfies \eqref{menos} then $u$ minimizer the energy functional $\Ff$. We define
$$
T(v)=\frac{\left\langle A(x)X v(x),Xv(x)\right\rangle^{\frac{p(x)}{2}}}{p(x)}+\frac{|v(x)\omega(x)|^{p(x)}}{p(x)}-f(x)v(x),
$$
then $T$ is a strictly convex functional and
$$
T(v_2+t(v_1-v_2))<(1-t)T(v_2)+tT(v_1),
$$
for $0<t<1$, we take $v_2=u$ and $v_1-v_2=v$ with $v \in W^{1,p(\cdot)}_{\omega,0}(\Omega)$,
\begin{align*}
T(u+tv)-T(u)&<t\left(T(v+u)-T(u)\right)\\
\frac{T(u+tv)-T(u)}{t}&<T(v+u)-T(u).
\end{align*}
Letting $t\to 0$ this yields by \eqref{deri}
\begin{align*}
T(v+u)-T(u)&>\left\langle A(x)Xu(x),Xu(x)\right\rangle ^{\frac{p(x)-2}{2}}\left\langle A(x)Xu(x),Xv(x)\right\rangle\\ \nonumber
	&\quad\qquad+|u(x)|^{p(x)-2}u(x)v(x)\omega(x)^{p(x)}-f(x)v(x).
\end{align*}
By \eqref{menos} and integrating in $\Omega$ we obtain that $\Ff(v+u)\geq \Ff(u)$ for all $v \in W^{1,p(\cdot)}_{\omega,0}(\Omega)$, i.e. $u$ minimizer the energy functional $\Ff$.
\end{proof}

\begin{remark}\label{existencia}
If we will consider $-1<t<0$ in \eqref{sol} and \eqref{deri} we get that if  $u$ minimizer the energy functional $\Ff$ then $u$ is a weak solution of the problem \eqref{1}.
 \end{remark}

We will use variational methods to prove the Theorem \ref{ext-unc}, we need the next classical result  
\begin{theorem}[Corollary 1.2.5, \protect\cite{AA}]\label{ambro-arco}
If $\mathcal{B}$ a reflexive Banach space, $A$ is a weakly closed subset in $%
\mathcal{B}$ and $\mathcal{F}:A\to\mathbb{R}$ is a weakly lower
semicontinuous, coercive functional in $A$, then there exists $u\in A$ such
that 
\begin{equation*}
\mathcal{F}(u)=\min_{v\in A}\mathcal{F}(v).
\end{equation*}
\end{theorem}

\begin{proof}[Proof Theorem \ref{ext-unc}]
Let us first observe that for the Remark \ref{existencia} and since the functional $\Ff$ is strictly convex if $u\in W^{1,p(\cdot)}_{\omega,0}(\Omega)$ minimizer the energy functional $\Ff$ then
$u$ is a unique weak solution of the problem \eqref{1}.

The space $W^{1,p(\cdot)}_{\omega,0}(\Omega)$ is a reflexive Banach space,  by Theorem \ref{ambro-arco} we have to prove that $\mathcal{F}:\,W^{1,p(\cdot)}_{\omega,0}(\Omega)\to\mathbb{R}$ is a weakly lower semicontinuous and coercive functional in $W^{1,p(\cdot)}_{\omega,0}(\Omega)$.

First we see that $\mathcal{F}$ is a coercive functional, let $u_{n}\in W_{\omega,0}^{1,p(\cdot )}(\Omega )$ such that 
$\Vert u_{n}\Vert _{W_{\omega,0}^{1,p(\cdot )}(\Omega )}\rightarrow \infty $. By Theorem \ref{HI} (H\"{o}lder inequality) and the hypothesis \eqref{A} we have 
\begin{align*}
\mathcal{F}(u_{n})& =\int_{\Omega }\frac{\left\langle A(x)Xu_{n}(x),Xu_{n}(x)\right\rangle ^{\frac{p(x)}{2}}}{p(x)}\,dx+\int_{\Omega }\frac{|u_{n}(x)\omega(x)|^{p(x)}}{p(x)}\,dx-\int_{\Omega
}f(x)u_{n}(x)\,dx \\
%& \geq \int_{\Omega }\frac{\left\langle A(x)Xu_{n}(x),Xu_{n}(x)\right\rangle^{\frac{p(x)}{2}}}{p(x)}\,dx-c\Vert f\Vert _{L^{p'(\cdot )}_{\omega }(\Omega)}\Vert u_{n}\Vert _{L_{\omega }^{p(\cdot )}(\Omega)} \\
& \geq \frac{C(A,p)}{p^{+}}\int_{\Omega
}|Xu_{n}|^{p(x)}\omega(x)^{p(x)}\,dx-c\Vert f\Vert _{L^{p'(\cdot )}_{\omega }(\Omega)}\Vert Xu_{n}\Vert _{L_{\omega }^{p(\cdot )}(\Omega)},
\end{align*}%
where $C(A,p)=\inf\{\eta_1^{p(x)/2}:x\in\Omega\}$.

By Theorem \ref{Poincare} and \ref{Poincare2}, (Poincar\'{e} inequality),  $\Vert u_{n}\Vert _{W_{\omega ,0}^{1,p(\cdot )}(\Omega )}\leq C \Vert Xu_{n}\Vert _{L_{\omega }^{p(\cdot )}(\Omega)}$.
As $\Vert u_{n}\Vert _{W_{\omega ,0}^{1,p(\cdot )}(\Omega )}\rightarrow \infty $ then $\Vert Xu_{n}\Vert _{L_{\omega}^{p(\cdot )}(\Omega)}\rightarrow \infty $ and $\Vert Xu_{n}\Vert _{L_{\omega
}^{p(\cdot )}}>1$ from $n$ large enough, by Lemma \eqref{norma-modular} 
\begin{equation*}
\int_{\Omega }|Xu_{n}|^{p(x)}\omega (x)^{p(x)}\,dx\geq \Vert Xu_{n}\Vert_{L_{\omega }^{p(\cdot )}(\Omega)}^{p^{-}_\Omega}
\end{equation*}%
and 
\begin{equation*}
\mathcal{F}(u_{n})\geq \frac{C(A,p)}{p^+}\Vert Xu_{n}\Vert_{L_{\omega }^{p(\cdot )}(\Omega)}^{p^{-}_\Omega}-c\Vert f\Vert _{L_{\omega }^{p^{\prime }(\cdot
)}}\Vert Xu_{n}\Vert _{L_{\omega }^{p(\cdot )}},
\end{equation*}%
since $p^-_\Omega>1$ we have $\mathcal{F}(u_{n})\rightarrow \infty $ when $\Vert
u_{n}\Vert _{W_{\omega,0}^{1,p(\cdot )}(\Omega )}\rightarrow \infty $.

Now, we must prove that $\mathcal{F}$ is weakly lower semicontinuous
functional. Let $u_{n}\in W_{\omega ,0}^{1,p(\cdot )}(\Omega )$ such that $%
u_{n}\rightharpoonup u$, as $|t|^{p(x)}$ is convex in $t$ for $p(x)>1$ we
have, 
\begin{equation*}
|u_{n}|^{p(x)}>|u(x)|^{p(x)}+p(x)|u|^{p(x)-2}u(u_{n}-u),
\end{equation*}%
similary, as $h:\mathbb{R}^{n_{1}}\rightarrow \mathbb{R}$ define by $h(\xi
)=\left\langle A(x)\xi ,\xi \right\rangle ^{{\frac{p(x)}{2}}}$ is convex in $\xi$ and  we have 
\begin{align*}
& \left\langle A(x)Xu_{n}(x),Xu_{n}(x)\right\rangle ^{{\frac{p(x)}{2}}%
}>\left\langle A(x)Xu(x),Xu(x)\right\rangle ^{{\frac{p(x)}{2}}} \\
& \qquad \qquad +{p(x)}\left\langle \left\langle
A(x)Xu(x),Xu(x)\right\rangle ^{{\frac{p(x)-2}{2}}}A(x)Xu(x),Xu_{n}(x)-Xu%
\right\rangle 
\end{align*}%
Then 
\begin{align}
\mathcal{F}(u_{n})& =\int_{\Omega }\frac{\left\langle
A(x)Xu_{n}(x),Xu_{n}(x)\right\rangle ^{\frac{p(x)}{2}}}{p(x)}%
\,dx+\int_{\Omega }\frac{|u_{n}(x)\omega(x)|^{p(x)}}{p(x)}\,dx-\int_{\Omega
}f(x)u_{n}(x)\,dx  \label{wlsc} \\
& \geq \int_{\Omega }\frac{\left\langle A(x)Xu(x),Xu(x)\right\rangle ^{\frac{%
p(x)}{2}}}{p(x)}\,dx+\int_{\Omega }\frac{|u(x)\omega(x)|^{p(x)}}{p(x)}%
\,dx-\int_{\Omega }f(x)u(x)\,dx  \notag \\
& \quad +\int_{\Omega }\left\langle \left\langle
A(x)Xu(x),Xu(x)\right\rangle ^{{\frac{p(x)-2}{2}}}A(x)Xu(x),Xu_{n}(x)-Xu%
\right\rangle \,dx  \notag \\
& \quad +\int_{\Omega }\omega(x)^{p(x)}|u|^{p(x)-2}u(u_{n}-u)\,dx-\int_{\Omega }f(x)(u_{n}(x)-u(x))\,dx.  \notag
\end{align}%
By hypothesis $f\in L_{\omega ^{-1}}^{p^{\prime }(\cdot )}(\Omega )$, observe that $\omega^{p(\cdot)}|u|^{p(\cdot)-2}u\in L_{\omega ^{-1}}^{p^{\prime }(\cdot )}(\Omega )$ and $\left\langle
AXu,Xu\right\rangle^{{\frac{p(\cdot)-2}{2}}}|AXu|\in L_{\omega^{-1}}^{p^{\prime }(\cdot )}(\Omega )$. Let $\lambda =(\|Xu\|_{L^{p(\cdot)}_\omega(\Omega)}+1)^{p^{-}-1}$, by the hypothesis \eqref{A} we get 
\begin{align*}
&\rho_{p'(\cdot),\omega^{-1}}(\left\langle AXu,Xu\right\rangle ^{{\frac{p(\cdot)-2}{2}}}\left\vert AXu\right\vert/\lambda)\\
&\qquad \qquad = \int_{\Omega }\lambda ^{-p^{\prime }(x)} \left(\left\langle
A(x)Xu(x),Xu(x)\right\rangle ^{{\frac{p(x)-2}{2}}}\left\vert A(x)Xu(x)\right\vert
\right)^{p^{\prime }(x)}\omega ^{-p^{\prime }(x)}dx \\
& \qquad \qquad \leq \int_{\Omega }\frac{\left( (\eta _{2}\omega^{2}|Xu|^{2})^{{\frac{p(x)-2}{2}}}\eta _{2}\omega^{2}|Xu|\right) ^{p^{\prime }(x)}}
{(\|Xu\|_{L^{p(\cdot)}_\omega(\Omega)}+1)^{(p^{-}-1)p^{\prime }(x)}}\omega^{-p^{\prime}(x)}dx \\
& \qquad \qquad \leq \int_{\Omega }\eta_2^{p(x)p'(x)}\frac{\left\vert Xu\right\vert ^{(p(x)-1)p^{\prime }(x)}}
{(\|Xu\|_{L^{p(\cdot)}_\omega(\Omega)}+1)^{p(x)}}\omega^{p(x)p^{\prime}(x)-p^{\prime }(x)}dx \\
& \qquad \qquad \leq C(A,p)\int_{\Omega }\frac{\left\vert Xu\right\vert ^{p(x)(x)} \omega^{p(x)}}
{(\|Xu\|_{L^{p(\cdot)}_\omega(\Omega)}+1)^{p(x)}}dx <C(A,p),
\end{align*}%
where $C(A,p)=\sup\{\eta_2^{p(x)p'(x)}:x\in\Omega\}$. Therefore $\left\langle AXu,Xu\right\rangle ^{{\frac{p(\cdot)-2}{2}}%
}|AXu|\in L_{\omega ^{-1}}^{p^{\prime }(\cdot )}(\Omega ).$

Now, we take limit on \eqref{wlsc}, as $u_n \rightharpoonup u$, we have 
\begin{align*}
\mathop{\underline{\lim}}_{n \to \infty}\mathcal{F}(u_n)&\geq \int_{\Omega} 
\frac{\left\langle A(x)X u(x),X u(x)\right\rangle^{\frac{p(x)}{2}}}{p(x)}
\,dx+\int_\Omega \frac{|u(x)\omega(x)|^{p(x)}}{p(x)}\,dx-\int_\Omega f(x)u(x)\,dx \\
&\quad+ \mathop{\underline{\lim}}_{n \to
\infty}\int_{\Omega}\left\langle\left\langle A(x)X u(x),X
u(x)\right\rangle^{ {\frac{p(x)-2}{2}}} A(x)X u(x),X
u_n(x)-Xu\right\rangle\,dx \\
&\quad +\mathop{\underline{\lim}}_{n \to \infty}\int_{\Omega} \omega(x)^{p(x)}|u|^{p(x)-2}u(u_n-u)\,dx-\mathop{\underline{\lim}}_{n \to
\infty}\int_\Omega f(x)(u_n(x)-u(x))\,dx \\
&=\quad\int_{\Omega}\frac{\left\langle A(x)X u(x),X u(x)\right\rangle^{\frac{
p(x)}{2}}}{p(x)}\,dx+\int_\Omega \frac{|u(x)\omega(x)|^{p(x)}}{p(x)}\,dx-\int_\Omega
f(x)u(x)\,dx \\
&=\mathcal{F}(u),
\end{align*}
as we wanted to show.
\end{proof}

\textbf{Acknowledgement.} We want to thank to Dr. Uriel Kaufmann, for his generosity and knowledge given to the group of Analysis and Differential Equations of the
FaMAF, Universidad Nacional de C\'ordoba.

\end{document}